\documentclass[10pt]{book}
\def\Bshorttitle{STEADY BOLTZMANN AND NAVIER-STOKES}
\def\Bauthor{K. AOKI, F. GOLSE AND S. KOSUGE}
\usepackage{hyperref}
\usepackage{amsmath}
\usepackage{amssymb}
\usepackage{amsfonts}
\usepackage{bm}
\usepackage{amsthm}
\usepackage{array,epsfig,fancyheadings}
\usepackage[sectionbib,numbers]{natbib}
\def\BS #1{\pbm[1]{#1}{#1}
\vskip .7cm \centerline{\bf #1}\par \vskip .3cm}
\def\BSs #1#2{\pbm[1]{#1}{#1}
\vskip .7cm \centerline{\bf #1}\par \vskip 5pt\noindent{\bf #2}\par\vskip .3cm}
\def\bs #1{\pbm[2]{#1}{#1}
\vskip .7cm\noindent{\bf #1}\par \vskip .3cm}
\skip\footins =.6cm
\def\pbm{\pdfbookmark}

\begin{document}
\allowdisplaybreaks\pagestyle{fancy} \thispagestyle{fancyplain}
\lhead[\fancyplain{}\leftmark]{} \chead[]{}
\rhead[]{\fancyplain{}\rightmark} \cfoot{\rm\thepage}
\def\n{\noindent}
\newtheorem{theorem}{Theorem}
\newtheorem{lemma}{Lemma}
\newtheorem{corollary}{Corollary}
\newtheorem{proposition}{Proposition}
\theoremstyle{definition}
\newtheorem{definition}{Definition}
\newtheorem{example}{Example}
\newtheorem{remark}{Remark}
\newcount\fs
\def\dd{\kern 2pt\raise 3pt\hbox{\large .}\kern 2pt\ignorespaces}
\def\fontsizefi{\fontsize{5}{9.5pt plus.8pt minus .6pt}\selectfont}
\def\fontsizesi{\fontsize{6}{10.5pt plus.8pt minus .6pt}\selectfont}
\def\fontsizese{\fontsize{7}{10.5pt plus.8pt minus .6pt}\selectfont}
\def\fontsizeei{\fontsize{8}{12pt plus.8pt minus .6pt}\selectfont}
\def\fontsizeni{\fontsize{9}{13.5pt plus1pt minus .8pt}\selectfont}
\def\fontsizete{\fontsize{10}{14pt plus.8pt minus .6pt}\selectfont}
\def\fontsizeel{\fontsize{10.95}{15pt plus2pt minus .0pt}\selectfont}
\def\fontsizetw{\fontsize{12}{116.5pt plus1pt minus .8pt}\selectfont}
\def\fontsizefo{\fontsize{14.4}{18pt plus1pt minus .8pt}\selectfont}
\def\fontsizesn{\fontsize{17.28}{20pt plus1pt minus 1pt}\selectfont}
\def\fontsizetn{\fontsize{20.73}{26pt plus1pt minus 1pt}\selectfont}
\def\fontsizetf{\fontsize{24.88}{33pt plus1.5pt minus 1pt}\selectfont}
\def\sz#1#2{\fs=#1#2
   \ifnum\fs=05 \fontsizefi
   \else\ifnum\fs=06 \fontsizesi
   \else\ifnum\fs=07 \fontsizese
   \else\ifnum\fs=08 \fontsizeei
   \else\ifnum\fs=09 \fontsizeni
   \else\ifnum\fs=10 \fontsizete
   \else\ifnum\fs=11 \fontsizeel
   \else\ifnum\fs=12 \fontsizetw
   \else\ifnum\fs=14 \fontsizefo
   \else\ifnum\fs=17 \fontsizesn
   \else\ifnum\fs=20 \fontsizetn
   \else\ifnum\fs=25 \fontsizetf
\fi\fi \fi \fi \fi \fi \fi \fi \fi \fi \fi \fi }

\markright{{\scriptsize $\ $}\hfill\break\vskip -16pt {\scriptsize $\ $}\hfill\break\vskip -16pt {\scriptsize $\ $} \hfill}
\markboth {\hfill{\small\rm\Bauthor}\hfill} {\hfill {\small\rm\Bshorttitle}\hfill}

\renewcommand{\thefootnote}{}
\abovedisplayskip 12pt
\belowdisplayskip 12pt
\abovedisplayshortskip 5pt
\belowdisplayshortskip 9pt


\def\bbb {\mathbf{b}}
\def\bC {\mathbf{C}}
\def\bD {\mathbf{D}}
\def\bF {\mathbf{F}}
\def\bG {\mathbf{G}}
\def\bh {\mathbf{h}}
\def\bK {\mathbf{K}}
\def\bN {\mathbf{N}}
\def\bP {\mathbf{P}}
\def\bQ {\mathbf{Q}}
\def\bR {\mathbf{R}}
\def\bS {\mathbf{S}}
\def\bT {\mathbf{T}}
\def\bz {\mathbf{z}}
\def\bZ {\mathbf{Z}}
\def\ff {\mathbf{f}}

\def\fA {\mathfrak{A}}
\def\fH {\mathfrak{H}}
\def\fS {\mathfrak{S}}
\def\fZ {\mathfrak{Z}}

\def\cA {\mathcal{A}}
\def\cB {\mathcal{B}}
\def\cC {\mathcal{C}}
\def\cD {\mathcal{D}}
\def\cE {\mathcal{E}}
\def\cF {\mathcal{F}}
\def\cG {\mathcal{G}}
\def\cH {\mathcal{H}}
\def\cJ {\mathcal{J}}
\def\cK {\mathcal{K}}
\def\cL {\mathcal{L}}
\def\cM {\mathcal{M}}
\def\cN {\mathcal{N}}
\def\cP {\mathcal{P}}
\def\cQ {\mathcal{Q}}
\def\cR {\mathcal{R}}
\def\cS {\mathcal{S}}
\def\cT {\mathcal{T}}
\def\cX {\mathcal{X}}
\def\cZ {\mathcal{Z}}

\def\scrH{\mathscr{H}}
\def\scrL{\mathscr{L}}

\def\a {{\alpha}}
\def\b {{\beta}}
\def\g {{\gamma}}
\def\Ga {{\Gamma}}
\def\de {{\delta}}
\def\eps {{\epsilon}}
\def\th {{\theta}}
\def\io {{\iota}}
\def\ka {{\kappa}}
\def\l {{\lambda}}
\def\L {{\Lambda}}
\def\si {{\sigma}}
\def\Si {{\Sigma}}
\def\Ups {{\Upsilon}}
\def\tUps {{\tilde\Upsilon}}
\def\om {{\omega}}
\def\Om {{\Omega}}

\def\d {{\partial}}
\def\grad {{\nabla}}
\def\Dlt {{\Delta}}

\def\rstr {{\big |}}
\def\indc {{\bf 1}}

\def\la {\langle}
\def\ra {\rangle}
\def \La {\bigg\langle}
\def \Ra {\bigg\rangle}
\def \lA {\big\langle \! \! \big\langle}
\def \rA {\big\rangle \! \! \big\rangle}
\def \LA {\bigg\langle \! \! \! \! \! \; \bigg\langle}
\def \RA {\bigg\rangle \! \! \! \! \! \; \bigg\rangle}

\def\pp {{\parallel}}

\def\vide {{\varnothing}}

\def\wL {{w-L}}
\def\wto {{\rightharpoonup}}

\def\Sh {{\hbox{S\!h}}}
\def\Fr {{\hbox{Fr}}}
\def\Kn{{\hbox{K\!n}}}
\def\Ma{{\hbox{M\!a}}}
\def\Rey{{\hbox{R\!e}}}

\def\bFc {{\bF^{\hbox{conv}}_\eps}}
\def\bFcn{{\bF^{\hbox{conv}}_{\eps_n}}}
\def\bFd {{\bF^{\hbox{diff}}_\eps}}
\def\bFdn{{\bF^{\hbox{diff}}_{\eps_n}}}

\def\DDS {\mathcal{D}_{Slater}}
\def\De {{_\eta D}}


\newcommand{\vvec}{\overrightarrow}
\newcommand{\Div}{\operatorname{div}}
\newcommand{\Rot}{\operatorname{curl}}
\newcommand{\Diam}{\operatorname{diam}}
\newcommand{\Dom}{\operatorname{Dom}}
\newcommand{\Sign}{\operatorname{sign}}
\newcommand{\Span}{\operatorname{span}}
\newcommand{\Supp}{\operatorname{supp}}
\newcommand{\Det}{\operatorname{det}}
\newcommand{\Tr}{\operatorname{trace}}
\newcommand{\Codim}{\operatorname{codim}}
\newcommand{\Dist}{\operatorname{dist}}
\newcommand{\DDist}{\operatorname{Dist}}
\newcommand{\Ker}{\operatorname{Ker}}
\newcommand{\Ran}{\operatorname{Ran}}
\newcommand{\Id}{\operatorname{Id}}
\newcommand{\Lip}{\operatorname{Lip}}

\newcommand{\ba}{\begin{aligned}}
\newcommand{\ea}{\end{aligned}}

\newcommand{\be}{\begin{equation}}
\newcommand{\ee}{\end{equation}}

\newcommand{\lb}{\label}


\noindent{Bulletin of the Institute of Mathematics}\vskip -6pt
\noindent{Academia Sinica (New Series)}

$\ $\par \vskip .5cm \centerline{\large\bf THE STEADY BOLTZMANN} \smallskip
\centerline{\large\bf  AND NAVIER-STOKES EQUATIONS}\vskip .5cm

\centerline{KAZUO AOKI$^{1,a}$, FRAN\c{C}OIS GOLSE$^{2,b}$ AND SHINGO KOSUGE$^{1,c}$}
\vskip .5cm
{\sz08
\n $^1$Department of Mechanical Engineering and Science, Kyoto University, Kyoto 615-8540, Japan  \vskip -4pt
\n $^a$E-mail: \href{mailto:aoki.kazuo.7a@kyoto-u.ac.jp}{aoki.kazuo.7a@kyoto-u.ac.jp} \vskip -4pt
\n $^2$Centre de math\'ematiques Laurent Schwartz, Ecole polytechnique, 91128 Palaiseau cedex, France \vskip -4pt
\n $^b$E-mail: \href{mailto:francois.golse@polytechnique.edu}{francois.golse@polytechnique.edu} \vskip -4pt
\n $^c$E-mail: \href{mailto:kosuge.shingo.6r@kyoto-u.ac.jp}{kosuge.shingo.6r@kyoto-u.ac.jp}
}

\bigskip
\rightline{\it To our friend and colleague Prof. Tai-Ping Liu on his 70th birthday}
\bigskip
\bigskip
\centerline{\bf Abstract}\par \vskip .2cm 
The paper discusses the similarities and the differences in the mathematical theories of the steady Boltzmann and incompressible Navier-Stokes equations posed in a bounded domain. First we discuss two different scaling limits in which
solutions of the steady Boltzmann equation have an asymptotic behavior described by the steady Navier-Stokes Fourier system. Whether this system includes the viscous heating term depends on the ratio of the Froude number to the 
Mach number of the gas flow. While the steady Navier-Stokes equations with smooth divergence-free external force always have at least one smooth solutions, the Boltzmann equation  with the same external force set in the torus, or in a 
bounded domain with specular reflection of gas molecules at the boundary may fail to have any solution, unless the force field is identically zero. Viscous heating seems to be of key importance in this situation. The nonexistence of any
steady solution of the Boltzmann equation in this context seems related to the increase of temperature for the evolution problem, a phenomenon that we have established with the help of numerical simulations on the Boltzmann equation
and the BGK model.
\par

\sz11



\BS{Introduction}

The Boltzmann equation and its hydrodynamic limits have been widely studied in the time-dependent regime. The Cauchy problem for the Boltzmann equation is discussed in \cite{Ukai1974,KanielShin,IllnerShin,DiPernaPLL}. Hydrodynamic
limits of the Boltzmann equation are analyzed by various methods in \cite{Nishida78,Caflisch80,BGL0,DMEL,BGL1,BGL2,BardosUkai,BGL3,PLLMasmoudi1,PLLMasmoudi2,GL2002,LSRBoltzEuler,GSR2004,GSR2009,LvrmrMasmou} --- see 
also \cite{VillaniBourbaki} for a nice introduction to the mathematical analysis of hydrodynamic limits of the Boltzmann equation. 

Some of the results known in the case of the Cauchy problem set on a spatial domain which either the Euclidean space $\bR^3$ or the periodic box $\bT^3$ have been extended to the case of the initial-boundary value problem in a domain
$\Om$ of $\bR^3$: see \cite{Mischler,MasmouLSR,BardosFGPaillard,FGCamwa} --- see also \cite{SRBook} for a more synthetic presentation of this material as well as further results.

By comparison, the mathematical literature on the analogous problems in the case of steady solutions is much more scarce. A great collection of asymptotic and numerical results on the Boltzmann equation in the steady regime can be found
in the books \cite{SoneBook1,SoneBook2}. As for the mathematical analysis of the boundary value problem for the Boltzmann equation, the main references are \cite{Guiraud1,Guiraud2} (see also \cite{GuiraudICM}), together with the more
modern references \cite{ArkeNouri,GuoEspoMarra}. Steady solutions of the Boltzmann equation for a gas flow past an obstacle have been investigated in detail in \cite{UkaiAsano1,UkaiAsano2}.

There are striking analogies between the Boltzmann and the incompressible Navier-Stokes equations in three space dimensions --- in the words of P.-L. Lions \cite{PLLKyoto2} ``[...] the global existence result of [renormalized] solutions [...] can 
be seen as the analogue for Boltzmann's equation to the pioneering work on the Navier-Stokes equations by J. Leray''. This analogy is at the origin of the program outlined in \cite{BGL0,BGL1,BGL2} and carried out in \cite{GSR2004,GSR2009}
--- see also \cite{LvrmrMasmou} for the extension to a more general class of collision kernels. 

It has been known for a long time that the regularity theory of solutions of the incompressible Navier-Stokes equations is much easier in the steady than in the time-dependent regime. For instance, in space dimension $3$, steady solutions of 
the Dirichlet problem for the incompressible Navier-Stokes equations have the same regularity as their boundary data and the external force field driving them: see Proposition 1.1 and Remark 1.6 in  chapter II, \S 1 of \cite{Temam}. At variance, 
it is still unknown at the time of this writing whether Leray solutions of the Cauchy problem for the Navier-Stokes equations in space dimension $3$ propagate the regularity of their initial data --- see Problems A-B in \cite{Fefferman}.

One striking difference between the steady and the time-dependent problems for the incompressible Navier-Stokes equations is the uniqueness theory, and its relation to the regularity of solutions. Smooth solutions of the time-dependent 
incompressible Navier-Stokes equations in space dimension $3$ are known to be uniquely determined by their initial data (and driving force field) within the class of Leray weak solutions, a remarkable result proved by Leray himself (see \S 32 
in \cite{Leray34}). On the contrary, it can be proved that bifurcations do occur on the steady problem for the incompressible Navier-Stokes equations, leading to nonuniqueness results. Such a nonuniqueness result for the Taylor-Couette problem 
has been proved in \cite{Welte} --- see also Chapter II, \S 4 in \cite{Temam}. Analogous bifurcations for the Boltzmann equation have been observed numerically in \cite{SoneDoi}, and mathematically in \cite{ArkeNouriTC}.

These considerations suggest studying the mathematical theory of existence and regularity for steady solutions of the Boltzmann equation driven by an external force field. In particular, does the assumption of a steady regime simplify regularity 
issues, as in the case of the Navier-Stokes equations?

As we shall see, there are strong similarities between the Boltzmann and the Navier-Stokes equations in the steady regime. After reviewing the basic structure of the Boltzmann equation in section 1, we propose in section 2 a formal derivation 
of two variants of the incompressible Navier-Stokes-Fourier system from the Boltzmann equation under appropriate scaling assumptions. Section 3 discusses the steady Navier-Stokes and Boltzmann equation in the periodic setting, which
leads to a striking difference between both models. Section 4 provides a physical explanation for this difference, based on numericla simulations of the evolution problem. After a brief section 5 summarizing our conclusions, some computations
involving Gaussian averages of certain vector and tensor fields have been put together in an appendix.

Professor Tai-Ping Liu is at the origin of some of the most striking results on the mathematical analysis of the equations fluid dynamics --- his work \cite{Liu78} on the compressible Euler system, and his analysis of the stability of the Boltzmann
shock profile in collaboration with S.-H. Yu \cite{LiuYu04} have had a lasting impact on our field. We are pleased to offer him this modest contribution on the occasion of his 70th birthday.

\newpage

\setcounter{chapter}{1}
\setcounter{equation}{0} 

\BS{1. The Steady Boltzmann Equation with External Force Field}

The Boltzmann equation with external force field $f\equiv f(x)\in\bR^3$ is posed in a bounded, convex spatial domain $\Om\subset\bR^3$ with smooth boundary $\d\Om$. The outward unit normal vector field on $\d\Om$ is denoted by $n_x$. 
According to \S 1.9 in \cite{SoneBook2}, its dimensionless form is
\be\lb{DimLessB}
v\cdot\grad_xF+\frac{\Ma^2}{\Fr^2}f\cdot\grad_vF=\frac1{\Kn}\cC(F)\,,\qquad x\in\Om\,,\,\,v\in\bR^3\,,
\ee
where
$$
\cC(F):=\iint_{\bR^3\times\bS^2}(F'F'_*-FF_*)|(v-v_*)\cdot\om|dv_*d\om\,,
$$
with the usual notation 
$$
F=F(x,v)\,,\quad F_*=F(x,v_*)\,,\quad F'=F(x,v')\,,\quad\hbox{ and }F'_*=F(x,v'_*)\,,
$$
assuming that
$$
v'=v-(v-v_*)\cdot\om\om\,,\quad\hbox{ and }v'_*=v_*+(v-v_*)\cdot\om\om\,.
$$
The dimensionless numbers $\Ma$, $\Fr$ and $\Kn$ are respectively the Mach, Froude and Knudsen numbers. We recall the definitions of these dimensionless numbers:
$$
\Ma=\frac{U_0}{\sqrt{RT_0}}\,,\qquad\Fr=\frac{U_0}{\sqrt{F_0L_0}}\,,\qquad\Kn=\frac{\ell_0}{L_0}\,,
$$
where $R$ is the specific gas constant, henceforth set to one for simplicity. In these formulas, $U_0$, $T_0$, $F_0$, $L_0$ are respectively the reference speed, temperature and external force in the gas, while $L_0$ is the reference length 
scale and $\ell_0$ is the mean free path of the gas molecules at the reference state.

We recall that the local conservation laws of mass, momentum and energy for the collision integral are
$$
\int_{\bR^3}\cC(F)dv=\int_{\bR^3}v_i\cC(F)dv=\int_{\bR^3}|v|^2\cC(F)dv=0\quad\hbox{ a.e. on }\Om
$$
for $i=1,2,3$, provided that $F\ge 0$ is measurable and decays rapidly enough as $|v|\to\infty$ --- so that, for instance
$$
\iint_{\Om\times\bR^3}(1+|v|^3)(1+|f|)Fdvdx<\infty\,.
$$
(See for instance chapter I, \S 4, Corollary 1 in \cite{Glassey}.) This implies the local mass, momentum and energy balance identities, which are the differential identities
$$
\ba
\Div_x\int_{\bR^3}vFdv&=0\,,&&\qquad\hbox{(mass)}
\\
\Div_x\int_{\bR^3}v\otimes vFdv&=\frac{\Ma^2}{\Fr^2}f\int_{\bR^3}Fdv\,,&&\qquad\hbox{(momentum)}
\\
\Div_x\int_{\bR^3}v\tfrac12|v|^2Fdv&=\frac{\Ma^2}{\Fr^2}f\cdot\int_{\bR^3}vFdv\,,&&\qquad\hbox{(energy)}
\ea
$$
are satisfied by all the solutions of the Boltzmann equation having the decay property mentioned above. Notice that this property implies that
$$
R^2_n\int_\Om\int_{|v|=R_n}F|f|ds(v)dx\to 0
$$
for some sequence $R_n\to\infty$. (The notation $ds$ designates the surface element on the sphere of radius $R_n$ centered at the origin.)

Integrating further in $x$ and applying Green's formula leads to the identities:
$$
\ba
\int_{\d\Om}\int_{\bR^3}Fv\cdot n_xdvdS(x)&=0\,,&&\qquad\hbox{(mass)}
\\
\int_{\d\Om}\int_{\bR^3}vF(v\cdot n_x)dvdS(x)&=\frac{\Ma^2}{\Fr^2}\iint_{\Om\times\bR^3}fFdxdv\,,&&\qquad\hbox{(momentum)}
\\
\int_{\d\Om}\int_{\bR^3}\frac12|v|^2F(v\cdot n_x)dvdS(x)&=\frac{\Ma^2}{\Fr^2}\iint_{\Om\times\bR^3}v\cdot fFdxdv\,,&&\qquad\hbox{(energy)}
\ea
$$
where $dS$ is the surface element on $\d\Om$, which are the global balance laws of mass, momentum and energy.

Multiplying both sides of the Boltzmann equation by $\ln F+1$ leads to the identity
$$
v\cdot\grad_x(F\ln F)+\frac{\Ma^2}{\Fr^2}f\cdot\grad_v(F\ln F)=\frac1{\Kn}\cC(F)(\ln F+1)\,.
$$
Integrating in $v$ leads to the local form of Boltzmann's H theorem (see for instance chapter I, \S 4, Corollary 1 in \cite{Glassey})
$$
\Div_x\int_{\bR^3}vF\ln Fdv=\frac1{\Kn}\int_{\bR^3}\cC(F)\ln Fdv\le 0\,,
$$
assuming again that $F$ decays rapidly enough as $|v|\to\infty$ --- for instance
$$
\iint_{\Om\times\bR^3}(1+|v|)(1+|f|)F\ln Fdvdx<\infty\,,
$$
so that
$$
R_n\int_\Om\int_{|v|=R_n}F|\ln F||f|ds(v)dx\to 0
$$
for some sequence $R_n\to\infty$.

Integrating further in $x$, one obtains the global form of Boltzmann's H theorem
$$
\int_{\d\Om}\int_{\bR^3}F\ln Fv\cdot n_xdvdS(x)=\frac1{\Kn}\iint_{\Om\times\bR^3}\cC(F)\ln Fdvdx\le 0\,.
$$
Boltzmann's H theorem asserts that the inequality above is an equality if and only if $F$ is a local Maxwellian, i.e. is of the form
$$
F(x,v)=\cM_{(\rho(x),u(x),\th(x))}(v)\,,
$$
with the notation
\be\lb{Maxw}
\cM_{(\rho,u,\th)}=\frac{\rho}{(2\pi\th)^{3/2}}e^{-\frac{|v-u|^2}{2\th}}\,.
\ee
(See for instance chapter I, \S \S 5 and 7, Corollary 2 in \cite{Glassey}.) .


\setcounter{chapter}{2}                           
\setcounter{equation}{0} 

\BS{2. The Navier-Stokes Limit for the Boltzmann Equation}

In this section, we explain how two variants of the Navier-Stokes-Fourier system can be derived from the Boltzmann equation. The exposition is formal and follows the style adopted in \cite{BGL0,BGL1}. The difference between the limiting
systems comes from the different scaling assumptions on the Froude number. Specifically, we are concerned with those variants of the incompressible Navier-Stokes-Fourier system which do, or do not include the viscous heating term.
This particular feature of the fluid dynamic limit of the Boltzmann equation is of key importance for the discussion in the present work.


\bs{2.1. From Boltzmann to Navier-Stokes-Fourier}

In this section, we assume the following scaling, where $\eps>0$ is a small parameter:
\be\lb{Scal1}
\Ma=\Kn=\eps\,,\quad\hbox{ and }\Fr=\sqrt{\eps}\,.
\ee
Hence, the scaled Boltzmann equation takes the form
$$
v\cdot\grad_xF_\eps+\eps f\cdot\grad_vF_\eps=\frac1\eps\cC(F_\eps)\,,\qquad x\in\Om\,,\,\,v\in\bR^3\,.
$$
In addition, write the Helmholtz decomposition of the external force field as
$$
f(x)=-\grad\Phi(x)+\eps f_s(x)\,,\quad\hbox{ with }\Div f_s=0\,.
$$
(In other words, we assume that the divergence free component of the external force is small compared to its curl free component.) With this additional assumption, the scaled Boltzmann equation becomes
\be\lb{ScalB1}
v\cdot\grad_xF_\eps-\eps\grad\Phi(x)\cdot\grad_vF_\eps+\eps^2f_s(x)\cdot\grad_vF_\eps=\frac1\eps\cC(F_\eps)\,,\qquad x\in\Om\,,\,\,v\in\bR^3\,.
\ee

Seek $F_\eps$ in the form
\be\lb{ScalDist1}
F_\eps(x,v)=M(v)\left(Z_\eps e^{\eps\Phi(x)}+\eps g_\eps(x,v)\right)\,,
\ee
assuming that 
\be\lb{TMassFluct}
\iint_{\Om\times\bR^3}g_\eps Mdxdv=0\,,
\ee
with the notation
$$
M:=\cM_{(1,0,1)}\,,
$$
and
\be\lb{NormZ}
\frac1{Z_\eps}:=\frac1{|\Om|}\int_{\Om}e^{\eps\Phi(x)}dx\,.
\ee
In other words, the distribution function is sought in the form of a perturbation of the order of the Mach number $\Ma=\eps$ about the uniform Maxwellian $M$. Except for the scaling assumption on the external force, this is exactly the same scaling
assumption as in \cite{BGL0,BGL1}.

In terms of $g_\eps$, the scaled Boltzmann equation (\ref{ScalB1}) takes the form
\be\lb{ScalBoltzg}
\ba
\eps v\cdot\grad_xg_\eps-\eps^2M^{-1}\grad_x\Phi\cdot\grad_v(Mg_\eps)+\eps^2M^{-1}f_s(x)\cdot\grad_v\left(M(Z_\eps e^{\eps\Phi(x)}+\eps g_\eps)\right)
\\
=-Z_\eps e^{\eps\Phi(x)}\cL g_\eps+\eps\cQ(g_\eps,g_\eps)
\ea
\ee
where 
\be\lb{DefLQ}
\cL g:=-M^{-1}D\cC(M)\cdot(Mg)\,,\qquad\cQ(g,g):=M^{-1}\cC(Mg)\,,
\ee
are respectively the linearized collision integral at $M$ and the collision integral intertwined with the multiplication by $M$. (The notation $D\cC(M)\cdot(Mg)$ designates the differential of the collision integral $\cC$ evaluated at $M$, and applied 
to the variation $Mg$ of distribution function.) The quadratic operator $\cQ$ defines a unique bilinear symmetric operator, also denoted $\cQ$, by the polarization  formula
$$
\cQ(f,g):=\tfrac12(\cQ(f+g,f+g)-\cQ(f,f)-\cQ(g,g))\,.
$$
In other words,
$$
\cL g=\iint_{\bR^3\times\bS^2}(g+g_*-g'-g'_*)|(v-v_*)\cdot\om|M_*dv_*d\om\,,
$$
while
$$
\cQ(f,g)=\tfrac12\iint_{\bR^3\times\bS^2}(f'g'_*+f'_*g'-fg_*-f_*g)|(v-v_*)\cdot\om|M_*dv_*d\om\,.
$$

Henceforth, the integration with respect to the Gaussian weight $M$ is denoted as follows:
$$
\la\phi\ra:=\int_{\bR^3}\phi(v)M(v)dv\,.
$$

\begin{theorem}\lb{T-NSF}
Let $F_\eps$ be a family of solutions of the scaled Boltzmann equation (\ref{ScalB1}), whose relative fluctuation $g_\eps$ defined in (\ref{ScalDist1}) satisfies
$$
g_\eps\to g 
$$
weakly in $L^2(\Om\times\bR^3;(1+|v|^2)Mdvdx)$, and
$$
\cQ(g_\eps,g_\eps)\to\cQ(g,g)
$$
weakly in $L^1(\Om;L^2(\bR^3;(1+|v|^2)Mdv))$, while
$$
\eps\la g_\eps\grad\phi(v)\ra\to 0
$$
weakly in $L^2(\Om)$ for each $\phi\in L^2(\bR^3;(1+|v|^2)Mdv)$.

Then
$$
g(x,v):=\overline{\th}+u(x)\cdot v+\th(x)\tfrac12(|v|^2-5)
$$
where 
$$
\overline\th:=\frac1{|\Om|}\int_\Om\th(x)dx
$$
and $(u,\th)$ is a solution of the incompressible Navier-Stokes-Fourier system
\be\lb{NSF}
\left\{
\ba
{}&\Div u=0\,,
\\	\\
&\Div(u^{\otimes 2})+\grad p=\nu\Dlt u+\th\grad\Phi+f_s\,,
\\	\\
&\tfrac52\Div(u\th)=\ka\Dlt\th-u\cdot\grad\Phi\,.
\ea
\right.
\ee
The values of the viscosity $\nu$ and heat diffusivity $\ka$ are determined implicitly in terms of the collision integral, by formulas (\ref{FlaNu}) and (\ref{FlaKa}).
\end{theorem}

We recall the following fundamental result.

\begin{lemma}\lb{L-Hilb}
The operator $\cL$ is an unbounded self-adjoint operator on the Hilbert space $L^2(\bR^3;Mdv)$ with domain $L^2(\bR^3;(1+|v|^2)Mdv)$. Moreover
$$
\cL\ge 0\quad\hbox{Êand }\Ker\cL=\Span\{1,v_1,v_2,v_3,|v|^2\}\,.
$$
Finally, $\cL$ satisfies the Fredholm alternative: one has
$$
\Ran\cL=(\Ker\cL)^\perp\,.
$$
\end{lemma}

This has been proved by Hilbert in 1912 (see for instance \cite{Glassey}, chapter III, \S \S 4-5 and \cite{Cerci75}, chapter IV, \S 6).

\begin{proof}
The proof of Theorem \ref{T-NSF} is rather involved; we follow the discussion in \cite{BGL0,BGL1}.

\smallskip
\noindent
\textit{Step 1: asymptotic form of the number density fluctuations}

Assuming that $g_\eps\to g$ weakly in $L^2(\Om\times\bR^3;Mdxdv)$, we pass to the limit in the sense of distributions in both sides of the equality (\ref{ScalBoltzg}), and obtain
$$
\cL g=0\,.
$$
Thus $g$ is of the form
\be\lb{Limg}
g(x,v)=\rho(x)+u(x)\cdot v+\th(x)\tfrac12(|v|^2-3)\,.
\ee

\smallskip
\noindent
\textit{Step 2: divergence-free and hydrostatic conditions}

Multiplying both sides of the scaled Boltzmann equation by $\frac1\eps M$ and integrating in $v$ leads to
$$
\Div_x\la vg_\eps\ra=\eps\int_{\bR^3}\Div_v(M(g_\eps\grad\Phi(x)-(Z_\eps e^{\eps\Phi(x)}+\eps g_\eps)f_s(x)))dv=0\,.
$$
Passing to the limit in both sides of this equality in the sense of distributions as $\eps\to 0$, we get
\be\lb{Incompr}
\Div_xu=\Div_x\la vg\ra=0\,.
\ee

Mutiplying both sides of the scaled Boltzmann equation by $\frac1\eps Mv$ and integrating in $v$ leads to
$$
\Div_x\la v^{\otimes 2}g_\eps\ra=-\eps\grad\Phi(x)\la g_\eps\ra+\eps f_s(x)\la(Z_\eps e^{\eps\Phi(x)}+\eps g_\eps)\ra\,.
$$
Passing to the limit as $\eps\to 0$ shows that
\be\lb{PreBouss}
\grad_x(\rho+\th)=\Div_x\la v^{\otimes 2}g\ra=0\,.
\ee
Passing to the limit in (\ref{TMassFluct}) shows that
$$
\int_\Om\rho(x)dx=\int_\Om\la g\ra(x)dx=0\,,
$$
so that
\be\lb{Bouss}
\rho(x)+\th(x)=\overline{\th}\,,\quad\hbox{ with }\overline\th:=\frac1{|\Om|}\int_\Om\th(x)dx\,.
\ee

\smallskip
\noindent
\textit{Step 3: motion equation}

Mutiplying both sides of the scaled Boltzmann equation by $\frac1{\eps^2}Mv$ and integrating in $v$ shows that
$$
\ba
\Div_x\frac1\eps\la A(v)g_\eps\ra&+\grad_x\frac1\eps\la\tfrac13|v|^2g_\eps\ra
\\
=&-\grad\Phi\la g_\eps\ra
\\
&+f_s\la(Z_\eps e^{\eps\Phi}+\eps g_\eps)\ra
\\
\to&-\rho\grad\Phi(x)+f_s
\ea
$$
in the sense of distributions as $\eps\to 0$, with the notation
$$
A(v):=v^{\otimes 2}-\tfrac13|v|^2\,.
$$
Using (\ref{Bouss}), we recast the limit above as
\be\lb{MomCvg}
\Div_x\frac1\eps\la A(v)g_\eps\ra+\grad_x\frac1\eps\la\tfrac13|v|^2g_\eps\ra\to(\th-\overline{\th})\grad\Phi+f_s\,.
\ee

The elementary properties of the tensor field $A$ used in the paper are recalled in the Appendix. In particular, elementary computations and Hilbert's Lemma \ref{L-Hilb} show that 
$$
A\in(\Ker\cL)^\perp=\Ran(\cL)\,,
$$
so that there exists a unique tensor field denoted $\hat A$ such that
$$
\cL\hat A=A\quad\hbox{ and }\hat A\bot\Ker\cL\,.
$$
(See Lemma \ref{L-PtyAB}.) Thus
$$
\frac1\eps\la A(v)g_\eps\ra=\La\hat A(v)\frac1\eps\cL g_\eps\Ra\,.
$$

Returning to the scaled Boltzmann equation, we observe that
$$
\ba
Z_\eps e^{\eps\Phi(x)}\frac1\eps\cL g_\eps=&\cQ(g_\eps,g_\eps)-v\cdot\grad_xg_\eps
\\
&+\eps M^{-1}\grad_x\Phi\cdot\grad_v(Mg_\eps)
\\
&-\eps M^{-1}f_s(x)\cdot\grad_v(M(Z_\eps e^{\eps\Phi(x)}+\eps g_\eps))
\\
\to&\cQ(g,g)-v\cdot\grad_xg
\ea
$$
as $\eps\to 0$. 

At this point, we recall the following useful result.

\begin{lemma}\lb{L-QKerL}
For each $\phi,\psi\in\Ker\cL$, one has
$$
\cQ(\phi,\psi)=\tfrac12\cL(\phi\psi)\,.
$$
\end{lemma}

See formula (60) in \cite{BGL1} (one should notice the slightly different definitions of $\cL$ and $\cQ$ in formula (20) of \cite{BGL1}, which account for the different sign and normalizing factor $\tfrac12$).

Since $g(x,\cdot)\in\Ker\cL$, Lemma \ref{L-QKerL} implies that $\cQ(g,g)=\tfrac12\cL(g^2)$, so that
$$
Z_\eps e^{\eps\Phi(x)}\frac1\eps\cL g_\eps\to\tfrac12\cL(g^2)-v\cdot\grad_xg\,.
$$
On the other hand, (\ref{Limg}) and (\ref{Bouss}) imply that
\be\lb{Limg2}
g=\overline{\th}+u\cdot v+\th\tfrac12(|v|^2-5)
\ee
so that
$$
\ba
g^2=&A(u):A(v)+2\th u\cdot B(v)+\tfrac13|v|^2|u|^2+\tfrac14\th^2(|v|^2-5)^2
\\
&+\overline{\th}^2+2\overline{\th}u\cdot v+\overline{\th}\th(|v|^2-5)\,,
\ea
$$
while
$$
v\cdot\grad_xg=A(v):\grad_xu+B(v)\cdot\grad_x\th\,,
$$
since $u$ is divergence-free. Here, the notation $B(v)$ designates the vector field
$$
B(v):=\tfrac12(|v|^2-5)v\,,
$$
whose properties are recalled in the Appendix.

Therefore
$$
\ba
\frac1\eps\la A(v)g_\eps\ra=&\La Z_\eps^{-1}e^{-\eps\Phi(x)}\hat A(v)Z_\eps e^{\eps\Phi(x)}\frac1\eps\cL g_\eps\Ra
\\
\to&\la(\tfrac12 A(v)g^2-\hat A(v)v\cdot\grad_xg)\ra
\\
=&A(u):\la\tfrac12 A(v)^{\otimes 2}\ra
\\
&-\grad u:\la\hat A(v)\otimes A(v)\ra
\\
=&A(u)-\nu(\grad u+(\grad u)^T)\,,
\ea
$$
on account of (\ref{Incompr}) and of the identities
$$
\left\{
\ba
\la A_{ij}A_{kl}\ra&=\de_{ik}\de_{jl}+\de_{il}\de_{jk}-\tfrac23\de_{ij}\de_{kl}\,,
\\
\la\hat A_{ij}A_{kl}\ra&=\nu(\de_{ik}\de_{jl}+\de_{il}\de_{jk}-\tfrac23\de_{ij}\de_{kl})\,,
\ea
\right.
$$
with
\be\lb{FlaNu}
\nu:=\tfrac1{10}\la\hat A:A\ra\,.
\ee
(See statement (2) in Lemma \ref{L-IntPtyAAhatBBhat}.)

Therefore
$$
\Div(A(u)-\nu(\grad u+(\grad u)^T))+\grad q=(\th-\overline{\th})\grad\Phi+f_s\,,
$$
which can be recast as
$$
\Div(u\otimes u)-\nu\Dlt u+\grad p=\th\grad\Phi(x)+f_s(x)
$$
where $p=q-\tfrac13|u|^2+\overline{\th}\Phi$. This is precisely the motion equation in the Navier-Stokes-Fourier system. Notice that the term
$$
\frac1\eps\la\tfrac13|v|^2g_\eps\ra
$$
appearing on the left hand side of the equality (\ref{MomCvg}) does not converge in the sense of distributions in general, but its gradient does. Define
$$
T:=\lim_{\eps\to 0}\grad\frac1\eps\la\tfrac13|v|^2g_\eps\ra\,.
$$
For each compactly supported, divergence free test vector field $\xi\equiv\xi(x)$, one has
$$
\ba
\la T,\xi\ra&=\lim_{\eps\to 0}\int_{\bR^3}\xi\cdot\grad\left(\frac1\eps\la\tfrac13|v|^2g_\eps\ra\right)dx
\\
&=-\lim_{\eps\to 0}\int_{\bR^3}\Div\xi\left(\frac1\eps\la\tfrac13|v|^2g_\eps\ra\right)dx=0\,.
\ea
$$
By Theorem 17' in \cite{DR}, $T$ viewed as a $1$-current is homologous to $0$, which means precisely that $T=\grad q$ for some distribution $q$.

\smallskip
\noindent
\textit{Step 4: the heat conduction equation}

Next we explain how to derive the heat conduction equation. Mutiplying both sides of the scaled Boltzmann equation by $\frac1\eps M\tfrac12(|v|^2-5)$ and integrating in $v$ shows that
$$
\ba
\Div_x\frac1\eps\la B(v)g_\eps\ra=&-\grad\Phi(x)\cdot\la vg_\eps\ra
\\
&+f_s(x)\cdot\la v(Z_\eps e^{\eps\Phi(x)}+\eps g_\eps)\ra
\\
\to&-u\cdot\grad\Phi(x)\,.
\ea
$$
On the other hand
$$
\ba
\frac1\eps\la B(v)g_\eps\ra=&\La\hat B(v)\frac1\eps\cL g_\eps\Ra
\\
=&\La Z_\eps^{-1}e^{-\eps\Phi(x)}\hat B(v)Z_\eps e^{\eps\Phi(x)}\frac1\eps\cL g_\eps\Ra
\\
\to&\la\hat B(v)\left(\tfrac12\cL(g^2)-v\cdot\grad_xg\right)\ra
\\
=&\la\tfrac12 B(v)g^2-\hat B(v)v\cdot\grad_xg\ra
\\
=&\la B(v)^{\otimes 2}\ra\cdot u\th
\\
&-\la\hat B(v)\otimes B(v)\ra\cdot\grad\th
\\
=&\tfrac52 u\th-\ka\grad\th\,.
\ea
$$
Indeed,
$$
\left\{
\ba
\la B_iB_j\ra=\tfrac52\de_{ij}\,,
\\
\la\hat B_iB_j\ra=\ka\de_{ij}\,,
\ea
\right.
$$
with
\be\lb{FlaKa}
\ka=\la\hat B\cdot B\ra\,.
\ee
(See statement (1) in Lemma \ref{L-IntPtyAAhatBBhat}.)

Therefore
$$
\tfrac52\Div(u\th)-\ka\Dlt\th=-u\cdot\grad\Phi\,,
$$
which is precisely the heat conduction equation in the Navier-Stokes-Fourier system.
\end{proof}


\bs{2.2. From Kinetic Theory to Viscous Heating}

In the asymptotic Navier-Stokes-Fourier regime discussed above, the fluctuations of velocity field and of temperature are small and of the same order $O(\eps)$. In this case, the fluctuation of kinetic energy is negligible when compared 
to the fluctuation of internal energy. In order to keep both fluctuations small and of the same order, it is natural to scale the fluctuation of velocity field as $O(\eps)$, while the fluctuation of temperature should be of order $O(\eps^2)$.

At the level of the Boltzmann equation, this scaling assumption is obtained by choosing the distribution function of the form
\be\lb{ScalDist2}
F_\eps(x,v)=M(1+\eps g_\eps(x,v)+\eps^2h_\eps(x,v))
\ee
where
\be\lb{EvOdd}
g_\eps(x,v)=-g_\eps(x,-v)\,,\quad\hbox{ while }\quad h_\eps(x,v)=h_\eps(x,-v)
\ee
for a.e. $(x,v)$. Here, we assume that there is no conservative force, i.e. we take the potential $\Phi$ identically $0$. The total external force acting on the gas is therefore $f_s\equiv f_s(x)$ such that $\Div f_s=0$.

The dimensionless Boltzmann equation (\ref{DimLessB}) is scaled as follows:
\be\lb{Scal2}
\Ma=\Kn=\eps\,,\quad\hbox{ and }\Fr=1\,.
\ee
In other words, the scaled Boltzmann equation takes the form
\be\lb{ScalB2}
v\cdot\grad_xF_\eps+\eps^2f_s(x)\cdot\grad_vF_\eps=\frac1{\eps}\cC(F_\eps)\,.
\ee

The exposition in this section follows closely \cite{BLUY}, where the idea of the even-odd decomposition of the distribution function seems to have been used for the first time.

First, we express the local balance laws of mass, momentum and energy in terms of the fluctuations $g_\eps$ and $h_\eps$. The odd contributions of either $g_\eps$ or $h_\eps$ vanish after integration in $v$, so that
$$
\int_{\bR^3}F_\eps dv=1+\eps^2\la h_\eps\ra\,,\qquad\int_{\bR^3}vF_\eps dv=\eps\la vg_\eps\ra\,,
$$
while
$$
\ba
\int_{\bR^3}v\otimes vF_\eps dv&=I+\eps^2\la v\otimes vh_\eps\ra\,,
\\
\int_{\bR^3}v|v|^2F_\eps dv&=\eps\la v|v|^2g_\eps\ra\,.
\ea
$$
Hence, the local balance laws of mass, momentum and energy implied by the Boltzmann equation take the form
$$
\Div_x\la vg_\eps\ra=0\,,\qquad\hbox{ (mass) }
$$
while
$$
\Div_x\la v\otimes vh_\eps\ra=f_s+\eps^2f_s\la h_\eps\ra\,,\qquad\hbox{ (momentum) }
$$
and
$$
\Div_x\la v\tfrac12|v|^2g_\eps\ra=\eps^2f_s\cdot\la vg_\eps\ra\,.\qquad\hbox{ (energy) }
$$

Likewise, both sides of the Boltzmann equation are decomposed into even and odd components, observing that, for each rapidly decaying $F$
$$
\cC(F\circ R)=\cC(F)\circ R\,,\quad\hbox{Êfor all }R\in O_3(\bR)\,.
$$
Since the Maxwellian $M$ is a radial function, one has $M\circ R=M$, and therefore
$$
\cL(\phi\circ R)=(\cL\phi)\circ R\,,\qquad\cQ(\phi\circ R,\psi\circ R)=\cQ(\phi,\psi)\circ R
$$ 
for all $R\in O_3(\bR)$ and all rapidly decaying $\phi,\psi$. These identities are satisfied in particular for $R=-I$. Hence the even and odd components of $\cC(F_\eps)$ are respectively
$$
\ba
\cC(F_\eps)^{\hbox{Êeven}}&=-\eps^2M\cL h_\eps+\eps^2M\cQ(g_\eps,g_\eps)+\eps^4M\cQ(h_\eps,h_\eps)\,,
\\
\cC(F_\eps)^{\hbox{Êodd }}&=-\eps M\cL g_\eps+2\eps^3M\cQ(g_\eps,h_\eps)\,.
\ea
$$
Therefore, the scaled Boltzmann equation is equivalent to the system
\be\lb{ScalBoltzeo}
\left\{
\ba
{}&v\cdot\grad_xg_\eps+\eps^2M^{-1}f_s\cdot\grad_v(Mg_\eps)=-\cL h_\eps+\cQ(g_\eps,g_\eps)+\eps^2\cQ(h_\eps,h_\eps)\,,
\\	\\
&\eps^2v\cdot\grad_xh_\eps-\eps^2v\cdot f_s+\eps^4M^{-1}f_s\cdot\grad_v(Mh_\eps)=-\cL g_\eps+2\eps^2\cQ(g_\eps,h_\eps)\,.
\ea
\right.
\ee

\begin{theorem}\lb{T-NSFVH}
Let $F_\eps$ be a family of solutions of the scaled Boltzmann equation (\ref{ScalB2}), whose relative fluctuations $g_\eps,h_\eps$ defined in (\ref{ScalDist2})-(\ref{EvOdd}) satisfy
$$
g_\eps\to g\,,\qquad h_\eps\to h
$$
weakly in $L^2(\Om\times\bR^3;(1+|v|^2)Mdvdx)$, and
$$
\cQ(g_\eps,g_\eps)\to\cQ(g,g)\,,\qquad\cQ(g_\eps,h_\eps)\to\cQ(g,h)
$$
weakly in $L^1(\Om;L^2(\bR^3;(1+|v|^2)Mdv))$, while
$$
\eps^2\la h_\eps\grad\phi(v)\ra\to 0
$$
weakly in $L^2(\Om)$ for each $\phi\in L^2(\bR^3;(1+|v|^2)Mdv)$.

Then
\be\lb{Limg3}
g(x,v)=u(x)\cdot v\,,
\ee
while
\be\lb{Limh}
\ba
h(x,v)=&\tfrac12A(u(x)):A(v)-\hat A(v):\grad_xu(x)
\\
&+\rho(x)+(\th(x)+\tfrac13|u(x)|^2)\tfrac12(|v|^2-3)
\ea
\ee
where $(u,\th)$ is a solution of the incompressible Navier-Stokes-Fourier system with viscous heating
\be\lb{NSFVH}
\left\{
\ba
{}&\Div u=0\,,
\\	\\
&\Div(u^{\otimes 2})+\grad p=\nu\Dlt u+f_s\,,
\\	\\
&\tfrac52\Div(u\th)-u\cdot\grad_xp=\ka\Dlt\th+\tfrac12\nu\left|\grad_xu+(\grad_xu)^T\right|^2\,.
\ea
\right.
\ee
The values of the viscosity $\nu$ and heat diffusivity $\ka$ are determined implicitly in terms of the collision integral, by formulas (\ref{FlaNu}) and (\ref{FlaKa}), as in Theorem \ref{T-NSF}.
\end{theorem}

\begin{proof}
The proof follows more or less the same lines as that of Theorem \ref{T-NSF}; see also \cite{BLUY}.

\smallskip
\noindent
\textit{Step 1: asymptotic form of $g_\eps$ and divergence-free condition}

We first deduce from the second equation in (\ref{ScalBoltzeo}) and the assumption that $g_\eps\to g$ in the sense of distributions that
$$
\cL g=0\,.
$$
Hence $g(x,\cdot)\in\Ker\cL$ and $v\mapsto g(x,\cdot)$ is odd for a.e. $x$, so that $g$ is of the form (\ref{Limg3}). The local conservation of mass implies that
$$
0=\Div_x\la vg_\eps\ra\to\Div_x\la vg\ra=\Div_xu
$$
in the sense of distributions, so that
\be\lb{Incompr2}
\Div_xu=0\,.
\ee

\smallskip
\noindent
\textit{Step 2: asymptotic form of $h_\eps$ and divergence-free condition}

Next we deduce from the first equation in (\ref{ScalBoltzeo}) that
$$
h_\eps\to h\quad\hbox{ with }v\cdot\grad_xg=-\cL h+\cQ(g,g)\,.
$$
Since $g(x,\cdot)\in\Ker\cL$, applying Lemma \ref{L-QKerL} shows that
$$
\cL(h-\tfrac12g^2)=-v\cdot\grad_xg=-v\otimes v:\grad_xu=-A(v):\grad_xu\,,
$$
since $u$ is divergence free by (\ref{Incompr2}). Hence
$$
h(x,v)=\tfrac12g^2(x,v)-\hat A(v):\grad_xu+h_0(x,v)\,,
$$
with $h_0(x,\cdot)\in\Ker\cL$ for a.e. $x$. Since $h_0$ is even in $v$, it is of the form
\be\lb{h0}
h_0(x,v)=\varpi(x)+\th(x)\tfrac12(|v|^2-3)\,.
\ee
On the other hand
$$
\ba
g^2(x,v)&=(u(x)\cdot v)^2=u(x)\otimes u(x):v\otimes v
\\
&=A(u(x)):A(v)+\tfrac13|u(x)|^2|v|^2\,,
\ea
$$
so that $h$ is of the form 
\be\lb{Limh2}
\ba
h(x,v)=&\tfrac12A(u(x)):A(v)-\hat A(v):\grad_xu(x)
\\
&+\rho(x)+(\th(x)+\tfrac13|u(x)|^2)\tfrac12(|v|^2-3)\,,
\ea
\ee
with $\rho=\varpi+\tfrac16|u|^2$. 

\smallskip
\noindent
\textit{Step 3: motion equation}

Passing to the limit in the sense of distributions in the local balance law of momentum shows that
$$
\Div_x\la v\otimes vh\ra=f_s\,.
$$
We insert the expression for $h$ found above in this identity. Observe that
$$
\ba
\la v\otimes vh\ra=&\la A(v)(\tfrac12A(u(x)):A(v)-\hat A(v):\grad_xu(x))\ra
\\
&+\la\tfrac13|v|^2\left(\rho(x)+(\th(x)+\tfrac13|u(x)|^2)\tfrac12(|v|^2-3)\right)\ra I
\ea
$$
because $A$ and $\hat A\perp\Ker\cL$. By statements (1)-(2) in Lemma \ref{L-IntPtyAAhatBBhat} 
$$
\la A(v)(\tfrac12A(u(x)):A(v)-\hat A(v):\grad_xu(x))\ra=A(u(x))-\nu(\grad_xu+(\grad_xu)^T)\,,
$$
while
$$
\la\tfrac13|v|^2\left(\rho(x)+(\th(x)+\tfrac13|u(x)|^2)\tfrac12(|v|^2-3)\right)\ra=\rho(x)+\th(x)+\tfrac13|u(x)|^2\,.
$$
Hence
$$
\Div_x\left(A(u)-\nu(\grad_xu+(\grad_xu)^T)\right)+\grad_x(\rho+\th+\tfrac13|u|^2)=f_s\,,
$$
or equivalently
\be\lb{NSEq2}
\Div_x(u\otimes u)-\nu\Dlt_xu+\grad_xp=f_s
\ee
with $p=\rho+\th$.

\smallskip
\noindent
\textit{Step 4: heat equation}

Finally, we combine the local balance of mass and energy so that
$$
\Div_x\frac1{\eps^2}\la Bg_\eps\ra=f_s\cdot\la vg_\eps\ra\to f_s\cdot u
$$
in the sense of distributions. 

Next, we transform the left hand side of the equality above by exactly the same method as in the previous section:
$$
\ba
\frac1{\eps^2}\la Bg_\eps\ra&=\La\hat B\frac1{\eps^2}\cL g_\eps\Ra
\\
&=\la\hat B(2\cQ(g_\eps,h_\eps)-v\cdot\grad_xh_\eps-\eps^2M^{-1}f_s\cdot\grad_v(Mh_\eps))\ra
\\
&\to\la\hat B(2\cQ(g,h)-v\cdot\grad_xh)\ra\,.
\ea
$$
Notice that
$$
\int_{\bR^3}\hat B f_s\cdot\grad_vMdv=-\la\hat B f_s\cdot v\ra=0\,,
$$
because $\hat B\perp\Ker\cL$. Hence
$$
\Div_x\la\hat B(2\cQ(g,h)-v\cdot\grad_xh)\ra=f_s\cdot u\,,
$$
and we insert in this last expression the explicit formulas for $g$ and $h$. 

First
\be\lb{Qgh=}
\ba
2\cQ(g,h)=&2\cQ(u\cdot v,\tfrac12A(u):A(v)+\rho+(\th+\tfrac13|u|^2)\tfrac12(|v|^2-3))
\\
&-2\cQ(u\cdot v,\hat A(v):\grad_xu)\,.
\ea
\ee

Applying formula (2.18c) in \cite{BLUY} shows that
$$
\ba
2\cQ(u\cdot v,\tfrac12A(u):A(v)+\rho+(\th+\tfrac13|u|^2)\tfrac12(|v|^2-3))&
\\
=\cL(\th u\cdot B(v)+\tfrac13C(v):u\otimes u\otimes u)&\,,
\ea
$$
where
$$
C(v):=\tfrac12(v\otimes v\otimes v-3v\otimes I)\,.
$$
Hence
\be\lb{Qhg1}
\ba
2\la\hat B(v)\cQ(u\cdot v,\tfrac12A(u):A(v)+\rho+(\th+\tfrac13|u|^2)\tfrac12(|v|^2-3))\ra
\\
=\la\hat B(v)\cL(\th u\cdot B(v)+\tfrac13C(v):u\otimes u\otimes u)\ra
\\
=\la B(v)(\th u\cdot B(v)+\tfrac13C(v):u\otimes u\otimes u)\ra
\\
=\tfrac12(|u|^2+5\th)u
\ea
\ee
by Lemma \ref{L-IntBC} and statement (1) in Lemma \ref{L-IntPtyAAhatBBhat}.

By Proposition 2.6 in \cite{BLUY} and statements (1) and (3) of Lemma \ref{L-IntPtyAAhatBBhat},
\be\lb{Qhg2}
\ba
2\la\hat B(v)\cQ(u\cdot v,\hat A(v):\grad_xu)\ra=&\la(A(v)\cdot u)\hat A(v):\grad_xu\ra&
\\
&-\la\hat B(v)(u\cdot v)A(v):\grad_xu\ra&
\\
=&(\nu-\tfrac25\ka)(\grad_xu+(\grad_xu)^T)\cdot u&\,.
\ea
\ee

Finally, we compute
\be\lb{Qhg3}
\ba
\la\hat B(v)\otimes v h\ra=&\tfrac12\la\hat B(v)\otimes v\otimes A(v)\ra:A(u)
\\
&-\la\hat B(v)\otimes v\otimes\hat A(v)\ra:\grad_xu
\\
&+(\th+\tfrac13|u|^2)\la\hat B(v)\otimes B(v)\ra
\\
=&\tfrac25\ka A(u)-c(\grad_xu+(\grad_xu)^T)+\ka(\th+\tfrac13|u|^2)I
\ea
\ee
according to statements (1) and (3) of Lemma \ref{L-IntPtyAAhatBBhat}. Notice that 
$$
\rho\la\hat B\otimes v\ra=0\,,
$$
and that
$$
\la\hat B\otimes v\tfrac12(|v|^2-3)\ra=\la\hat B\otimes B\ra\,,
$$
because $\hat B\perp\Ker\cL$.

Therefore, putting together (\ref{Qhg1})-(\ref{Qhg2})-(\ref{Qhg3}), we arrive at the identity
\be\lb{EnergyBal2}
\ba
\Div_x\la\hat B(2\cQ(g,h)-v\cdot\grad_xh)\ra&
\\
=\Div_x(\tfrac12(|u|^2+5\th)u-(\nu-\tfrac25\ka)(\grad_xu+(\grad_xu)^T)\cdot u)&
\\
-\grad_x\otimes\grad_x:(\tfrac25\ka A(u)-c(\grad_xu+(\grad_xu)^T)+\ka(\th+\tfrac13|u|^2)I)&\,.
\ea
\ee

This identity can be substantially simplified, as follows. First,
$$
\grad_x\otimes\grad_x:\left(\grad_xu+(\grad_xu)^T\right)=0
$$
because of the divergence-free condition (\ref{Incompr2}). On the other hand,
$$
\Div_x\left((\grad_xu+(\grad_xu)^T)\cdot u)\right)=\grad_x\otimes\grad_x:u\otimes u+\Dlt_x\tfrac12|u|^2\,,
$$
while
$$
\ba
\grad_x\otimes\grad_x:(A(u)+\tfrac56|u|^2I)&=\grad_x\otimes\grad_x:u\otimes u-\tfrac13\Dlt_x|u|^2+\tfrac56\Dlt_x|u|^2
\\
&=\grad_x\otimes\grad_x:u\otimes u+\Dlt_x\tfrac12|u|^2\,.
\ea
$$
Hence
$$
\ba
\Div_x\la\hat B(2\cQ(g,h)-v\cdot\grad_xh)\ra&
\\
=\Div_x\left(\tfrac12(|u|^2+5\th)u-\nu\left(\grad_xu+(\grad_xu)^T\right)\cdot u\right)-\ka\Dlt_x\th&\,,
\ea
$$
so that the limiting form of the local energy balance is
$$
\Div_x\left(\tfrac12(|u|^2+5\th)u\right)=\nu\Div_x((\grad_xu+(\grad_xu)^T)\cdot u)+\ka\Dlt_x\th+f_s\cdot u\,.
$$
On the other hand, multiplying both sides of the Navier-Stokes motion equation by $u$, we arrive at the identity
$$
\ba
\Div_x\left(u\tfrac12|u|^2\right)+u\cdot\grad_xp&=f_s\cdot u+\nu u\cdot\Dlt_xu
\\
&=f_s\cdot u+\nu u\cdot\Div_x\left(\grad_xu+(\grad_xu)^T\right)\,.
\ea
$$
We further simplify the right hand side of the equality above as follows:
$$
\ba
u\cdot\Div_x\left(\grad_xu+(\grad_xu)^T\right)=&\Div_x\left(\left(\grad_xu+(\grad_xu)^T\right)u\right)
\\
&-\left(\grad_xu+(\grad_xu)^T\right):\grad_xu
\\
=&\Div_x\left(\left(\grad_xu+(\grad_xu)^T\right)u\right)
\\
&-\tfrac12\left|\grad_xu+(\grad_xu)^T\right|^2\,.
\ea
$$
Combining these identities with (\ref{EnergyBal2}) above leads to
\be\lb{TempVH}
\tfrac52\Div_x(u\th)-u\cdot\grad_xp=\ka\Dlt_x\th+\tfrac12\nu\left|\grad_xu+(\grad_xu)^T\right|^2\,.
\ee
\end{proof}

Observe that the Navier-Stokes motion equation is exactly the same when $\Phi=0$ for both scaling assumptions (\ref{Scal1}) and (\ref{Scal2}). The temperature equation, however, is very different according to whether the Froude
number is $O(\sqrt{\eps})$ as in (\ref{Scal1}), or $O(1)$ as in (\ref{Scal2}).

The viscous heating term on the right hand side of the equation governing the temperature field appears in Sone's asymptotic analysis of the hydrodynamic limits of the Boltzmann equation in the weakly nonlinear regime --- see 
\cite{SoneBook2}. Sone's original work \cite{Sone71} on the weakly nonlinear hydrodynamic limits of kinetic theory was written in the case of the BGK model; see \cite{SoneAoki87}Ê for the extension to the Boltzmann equation. 
Sone's argument is based on the Hilbert expansion. One should pay attention to a particular feature of Sone's theory: the pressure field $p$ in (\ref{NSF}) and the velocity and temperature fields $u$ and $\th$ in (\ref{NSF}) do not
appear at the same order in Sone's expansion; in fact $p$ appears at order $O(\eps^2)$ in the expansion of the distribution function in powers of $\eps$, while $u$ and $\th$ appear at order $O(\eps)$ in that same expansion.
See formulas (3.77), (3.79b-c), (3.80d) and (3.88a-c) in section 3.2.2 of \cite{SoneBook2}. The limit leading to (\ref{NSFVH}) corresponds to a situation where $u$ appears at order $O(\eps)$ in Sone's expansion, while the leading
order temperature fluctuation (denoted $\tau_{S1}$ in Sone's analysis, see formula (3.79c) in \cite{SoneBook2}) is identically zero. The temperature fluctuation appears at order $O(\eps^2)$ in Sone's expansion, together with 
the pressure field $p$, and the temperature equation in (\ref{NSFVH}) coincides with formula (3.89c) in section 3.2.2 of \cite{SoneBook2}. (Sone's analysis in \cite{SoneBook2} does not involve an external force, but the work of the
external force disappears from the temperature equation when combining the motion equation and the energy equation as explained above.)


\bs{2.3. Boundary Conditions}

The discussion of the incompressible Navier-Stokes-Fourier limit of the Boltzmann equation presented above would remain incomplete without discussing the boundary condition. In this section, we briefly describe the simplest 
imaginable situation.

Assume that the scaled Boltzmann equation (\ref{ScalB1}) is supplemented with a diffuse reflection condition at the boundary of the spatial domain $\Om$ on which the Boltzmann equation (\ref{DimLessB}) is posed. In other words,
for each $x\in\d\Om$, one has
\be\lb{DiffReflex}
F_\eps(x,v)=\sqrt{2\pi}M\int_{\bR^3}F_\eps(x,v)(v\cdot n_x)_+dv\,,\qquad v\cdot n_x<0\,,
\ee
where $x\mapsto n_x$ is the unit normal field defined on the boundary $\d\Om$ of the spatial domain. Here, we have assumed for simplicity that there is no temperature gradient on $\d\Om$. The constant temperature at the boundary 
(i.e. the temprature $1$ in the Maxwellian state $M$) defines the scale of the speed of sound in the interior of the domain. 

By construction
$$
\int_{\bR^3}F_\eps(x,v)v\cdot n_xdv=0\,,\qquad x\in\d\Om\,,
$$
which means that the net mass flux at each point $x\in\d\Om$ is identically $0$. This suggests that the boundary condition (\ref{DiffReflex}) should be supplemented with the additional condition
\be\lb{TMass}
\iint_{\Om\times\bR^3}F_\eps(x,v)dxdv=|\Om|\,,
\ee
that is consistent to leading order with the normalization of $Z_\eps$ in (\ref{ScalDist1})-(\ref{ScalDist2}), and is equivalent to the condition (\ref{TMassFluct}) already introduced above in the case $\Ma=\Kn=\Fr^2=\eps$, and
$$
\int_\Om\la h_\eps\ra dx=0
$$
in the case $\Ma=\Kn=\Fr=\eps$. (Notice that, in the latter case, $\la g_\eps\ra=0$ a.e. on $\Om$ since $g_\eps$ is odd in $v$.)

Besides, we assume that the force $f_s$ satisfies both
\be\lb{Condfs}
\Div_xf_s=0\quad\hbox{ on }\Om\,,\qquad f_s\cdot n_x=0\quad\hbox{ on }\d\Om\,.
\ee

Define
$$
\L_x(\phi):=\sqrt{2\pi}\la\phi(v\cdot n_x)_+\ra\,.
$$

\begin{theorem}
Let $\Ma=\Kn=\Fr^2=\eps$, and consider a family of solutions of the scaled Boltzmann equation (\ref{ScalB1}) supplemented with the diffuse reflection condition (\ref{DiffReflex}) and with the total mass condition (\ref{TMass}).
Assume that $F_\eps=M(Z_\eps e^{\eps\Phi}+\eps g_\eps)$ as in (\ref{ScalDist1}) and satisfies the same assumptions as in Theorem \ref{T-NSF}. Assume moreover that the family of traces of $g_\eps$ on the boundary 
$\d\Om\times\bR^3$ satisfies
$$
g_\eps\rstr_{\d\Om\times\bR^3}\to g\rstr_{\d\Om\times\bR^3}\quad\hbox{ weaklyÊ in }L^2(\d\Om\times\bR^3;|v\cdot n_x|MdvdS(x))
$$
where $g$ is such that $g_\eps\to g$ weakly in $L^2(\Om\times\bR^3,Mdvdx)$. Then
$$
g(x,v)=\overline{\th}+u(x)\cdot v+\th(x)\tfrac12(|v|^2-5)
$$
where 
$$
\overline{\th}=\frac1{|\Om|}\int_\Om\th(x)dx\,,
$$
and $(u,\th)$ is a solution of the Navier-Stokes-Fourier system (\ref{NSF}), with Diri- chlet boundary condition
$$
u\rstr_{\d\Om}=0\,,\qquad\th\rstr_{\d\Om}=0\,.
$$
\end{theorem}

\begin{proof}
Indeed, the diffuse reflection condition implies that
\be\lb{DiffReflexFluct}
g_\eps(x,v)=\L_x(g_\eps(x,\cdot))\,,\qquad v\cdot n_x<0\,.
\ee
Thus one can pass to the limit as $\eps\to 0$ in (\ref{DiffReflexFluct}). One arrives at
$$
g(x,v)=\L_x(g(x,\cdot))\,,\qquad x\in\d\Om\,,\,\,v\in\bR^3\,.
$$
Since we already know from Theorem \ref{T-NSF} that $g$ of the form
$$
g(x,v)=\overline{\th}+u(x)\cdot v+\th(x)\tfrac12(|v|^2-5)\,,
$$
this implies that
$$
u\rstr_{\d\Om}=0\,,\quad\hbox{ and }\th\rstr_{\d\Om}=0\,.
$$
\end{proof}

\begin{theorem}
Let $\Ma=\Kn=\eps$ while $\Fr=1$, and consider a family of solutions of the scaled Boltzmann equation (\ref{ScalB2}) supplemented with the diffuse reflection condition (\ref{DiffReflex}) and with the total mass condition (\ref{TMass}).
Assume that $F_\eps$ satisfies the same assumptions as in Theorem \ref{T-NSFVH}, and that the odd and even part of the relative fluctuation of distribution function, resp. $g_\eps$ and $h_\eps$ defined in (\ref{ScalDist2})
are continuous in $v$ and satisfy the condition
$$
g_\eps\rstr_{\d\Om\times\bR^3}\to g\rstr_{\d\Om\times\bR^3}\quad\hbox{Êand }\quad h_\eps\rstr_{\d\Om\times\bR^3}\to h\rstr_{\d\Om\times\bR^3}
$$
locally uniformly in $x,v$, where we recall that $g$ and $h$ are the weak limits of $g_\eps$ and $h_\eps$ in $L^2(\Om\times\bR^3;(1+|v|^2)Mdvdx)$ as $\eps\to 0$. Then $g$ and $h$ are given by the expressions (\ref{Limg3}) 
and (\ref{Limh2}), where $(u,\th)$ is a solution of the system (\ref{NSFVH}) with the Dirichlet boundary condition 
$$
u\rstr_{\d\Om}=0\,,\qquad\th\rstr_{\d\Om}=0\,.
$$
\end{theorem}

\begin{proof}

Specializing the equality above to the case where $v$ is tangential to the boundary, we find that
$$
g_\eps(x,v)=\L_x((g_\eps+\eps h_\eps)(x,\cdot))-\eps h_\eps(x,v)\,,\quad x\in\d\Om\,,\,\,v\cdot n_x=0\,,
$$
and observe that the left hand side of this equality is odd in $v=v-(v\cdot n_x)n_x$, while the right hand side is even. Therefore both sides vanish, so that 
$$
g_\eps(x,v)=0\,,\quad x\in\d\Om\,,\,\,v\cdot n_x=0\,.
$$
while
$$
h_\eps(x,v)=\L_x\left(\left(\frac1\eps g_\eps+h_\eps\right)(x,\cdot)\right)\,,\quad x\in\d\Om\,,\,\,v\cdot n_x=0\,.
$$
Passing to the limit on both sides of the first equality as $\eps\to 0$ shows that
$$
g(x,v)=\L_x(g(x,\cdot))\,,\quad x\in\d\Om\,,\,\,v\cdot n_x=0\,,
$$
and we conclude from (\ref{Limg3}) that
$$
u\rstr_{\d\Om}=0\,.
$$

Next we consider the differential operator
$$
P(x,D_v):=(I-n_x\otimes n_x)\grad_v
$$
--- which is the orthogonal projection of $\grad_v$ on the tangential direction of $\d\Om$ at $x$, and observe that
$$
P(x,D_v)h_\eps(x,v)=0\,,\quad x\in\d\Om\,,\,\,v\cdot n_x=0\,.
$$
Passing to the limit in both sides of this identity, we conclude that
$$
P(x,D_v)h(x,v)=0\,,\quad x\in\d\Om\,,\,\,v\cdot n_x=0\,.
$$
Substituting the expression (\ref{Limh2}) in this equality, we find that
$$
P(x,D_v)h(x,v)=-P(x,D_v)\hat A(v):\grad u(x)+v_\tau\th(x)\,,\quad x\in\d\Om\,,\,\,v\cdot n_x=0
$$
with
$$
v_\tau=v-(v\cdot n_x)n_x\,.
$$
Observe that
$$
P(x,D_v)\hat A(v):\grad u(x)=v_\tau\cdot\grad u(x)+(I-n_x\otimes n_x)\grad(u(x)\cdot v_\tau)=0\,,
$$
since we already know that $u\rstr_{\d\Om}=0$ and all the derivatives of $u$ appearing in the expression above are taken in directions tangential to $\d\Om$. Hence
$$
v_\tau\th(x)=0\,,\quad \hbox{ for each }x\in\d\Om\,,\,\,v\cdot n_x=0\,,
$$
which implies that $\th\rstr_{\d\Om}=0$.
\end{proof}


\setcounter{chapter}{3}                           
\setcounter{equation}{0} 

\BSs{3. Spatially Periodic Steady Solutions}{3.1. The Navier-Stokes Equations}

In this section, we assume that the spatial domain is $\Om=\bT^3$. Consider the system (\ref{NSF}) posed on $\bT^3$, and seek solutions satisfying
\be\lb{Mean0}
\int_{\bT^3}u(x)dx=\int_{\bT^3}\th(x)dx=0\,.
\ee
For simplicity, we assume further that $\Phi\equiv 0$. 

Multiplying both sides of the last equation in (\ref{NSF}) by $\th$, we see that
$$
u\cdot\grad(\tfrac12\th^2)=\Div(u\tfrac12\th^2)=\ka\th\Dlt\th\,,
$$
so that
$$
\ka\int_{\bT^3}|\grad\th(x)|^2dx=0\,.
$$
Therefore $\th\equiv 0$. 

Conversely, if $u$ is a solution of the motion equation and $\Phi\equiv 0$, then $(u,0)$ is a solution of the Navier-Stokes-Fourier system without viscous heating.

There exist indeed nontrivial solutions of the motion equation with nonzero solenoidal external force $f_s$. The simplest example is the case of a shear flow
$$
u(x)=(0,0,U(x_1,x_2))\,,\quad f_s(x)=(0,0,a(x_1,x_2))
$$
with
$$
\int_{\bT^2}U(x_1,x_2)dx_1dx_2=\int_{\bT^2}a(x_1,x_2)dx_1dx_2=0\,.
$$
Obviously
$$
\Div u=\d_{x_3}U=0\,,\qquad\Div f_s=\d_{x_3}a=0\,,
$$
and 
$$
\Div(u^{\otimes 2})=(u\cdot\grad)u=U\d_{x_3}u=0\,,
$$
so that the Navier-Stokes equation reduces to the Poisson equation in $\bT^2$:
$$
-\nu\Dlt U(x_1,x_2)=a(x_1,x_2)\,,\quad(x_1,x_2)\in\bT^2\,.
$$
For each zero-mean $a\in L^2(\bT^3)$, there exists a unique zero-mean $U\in H^2(\bT^2)$ of the Poisson equation above.

More generally, the following result is classical.

\begin{theorem} For each $f_s\in L^2(\bT^3)$ satisfying
$$
\int_{\bT^3}f_s(x)dx=0\,,
$$
there exists at least one solution $u\in H^2(\bT^3)$ of the Navier-Stokes equations with external force $f_s$ such that
$$
\int_{\bT^3}u(x)dx=0\,.
$$
Besides, there exists $c>0$ such that the solution $u\in H^2(\bT^3)$ and is unique if $\|f_s\|_{L^2}\le c\nu^2$.
\end{theorem}

The proof of this classical result is given below --- see chapter II, \S 1 in \cite{Temam} for a similar result in a slightly different (nonperiodic) setting. First, we recall some elements of notation. We denote by $\fH$ the subspace of 
$L^2(\bT^3;\bR^3)$ of vector fields $v$ such that
$$
\int_{\bT^3}v(x)dx=0\,,
$$
and we set $\fH^1=H^1(\bT^3)\cap\fH$. It will be convenient to use the norm 
$$
\|v\|_{\fH^1}:=\|\grad v\|_{L^2}\,.
$$

We denote by $\Pi$ the $L^2$-orthogonal projection on divergence-free vector fields. In other words, if $v\in \fH$, its Fourier decomposition is
$$
v(x)=\sum_{k\in\bZ^3\setminus\{0\}}v_ke^{i2\pi k\cdot x}\,,
$$
and
$$
\Pi v(x)=\sum_{k\in\bZ^3\setminus\{0\}}(v_k-|k|^{-2}k(k\cdot v_k))e^{i2\pi k\cdot x}\,.
$$
Likewise, for each zero-mean $\phi\in\fH$ we define 
$$
(-\Dlt)^{-s}\phi(x):=\sum_{k\in\bZ^3\setminus\{0\}}|2\pi k|^{-2s}\phi_ke^{i2\pi k\cdot x}\quad\hbox{Êwhere }\phi_k=\int_{\bT^3}\phi(x)e^{-i2\pi k\cdot x}dx\,.
$$

\begin{proof}
Consider the map 
$$
T:\,v\mapsto(-\Dlt)^{-1}\Pi\Div(v^{\otimes 2})\,.
$$
First, observe that $T$ maps $\fH^1$ into itself:
$$
\ba
\|Tv\|_{H^1}&=\|(-\Dlt)^{-1/2}\Pi\Div(v^{\otimes 2})\|_{L^2}=\|\Pi(-\Dlt)^{-1/2}\Div(v^{\otimes 2})\|_{L^2}
\\
&\le\|(-\Dlt)^{-1/2}\Div(v^{\otimes 2})\|_{L^2}\le\|v^{\otimes 2}\|_{L^2}\le\|v\|^2_{L^4}\le C^2\|v\|^2_{\fH^1}
\ea
$$
by Sobolev embedding ($\fH^1\subset L^p(\bT^3)$ for $1\le p\le 6$). Similarly
$$
\ba
\|Tu-Tv\|_{H^1}=\|(-\Dlt)^{-1/2}\Pi\Div(u^{\otimes 2}-v^{\otimes 2})\|_{L^2}\le\|u^{\otimes 2}-v^{\otimes 2}\|_{L^2}&
\\
\le\|u\otimes(u-v)\|_{L^2}+\|(u-v)\otimes v\|_{L^2}\le(\|u\|_{L^4}+\|v\|_{L^4})\|u-v\|_{L^4}&
\\
\le C^2(\|u\|_{\fH^1}+\|v\|_{\fH^1})\|u-v\|_{\fH^1}&\,.
\ea
$$
Hence $T$ is continuous from $\fH^1$ into itself, and Lipschitz continuous on balls of $\fH^1$.

The Navier-Stokes equations can be put in the form
$$
u=(-\nu\Dlt)^{-1}f_s-\tfrac1\nu Tu\,,
$$
and is embedded into the family of equations
$$
\cF_\a(u)=(-\nu\Dlt)^{-1}f_s\,,\quad\hbox{ where }\cF_\a=I+\frac\a\nu T
$$
parametrized by $\a\in[0,1]$.

If $u,v\in\fH^1$ with $\|u\|_{H^1}\le R$ and $\|v\|_{H^1}\le R$, then the map $\cG_\a$ defined by
$$
\cG_\a(w)=(-\nu\Dlt)^{-1}f_s+w-\cF_\a(w)
$$
satisfies the bound
$$
\|\cG_\a(u)-\cG_\a(v)\|_{H^1}\le\frac{2C^2R}{\nu}\|u-v\|_{H^1}
$$
for $\a\in[0,1]$. On the other hand, if $u\in\fH^1$ with $\|u\|_{H^1}\le R$
$$
\|\cG_\a(u)\|_{H^1}\le\|(-\nu^2\Dlt)^{-1/2}f_s\|_{L^2}+\frac{C^2}\nu R^2\le R\,.
$$
Thus, if 
$$
\|(-\Dlt)^{-1/2}f_s\|_{L^2}\le\tfrac12\nu R\quad\hbox{ and }R<\frac{\nu}{2C^2}\,,
$$
then $\cG_\a$ maps the closed ball $\overline{B(0,R)}\subset\fH_1$ into itself and is a strict contraction on $\overline{B(0,R)}$. Hence $\cG_\a$ has a unique fixed point in $\overline{B(0,R)}\subset\fH^1$
for $\a\in[0,1]$ provided that
$$
\|(-\Dlt)^{-1/2}f_s\|_{L^2}<\frac{\nu^2}{4C^2}\,,\quad\hbox{ with }R=\frac2\nu\|(-\Dlt)^{-1/2}f_s\|_{L^2}\,.
$$
In particular, for $\a=1$, this unique fixed point of $\cG_\a$ is the unique solution of the Navier-Stokes equation in $\overline{B(0,R)}\subset\fH^1$.

The estimate
$$
\|Tu-Tv\|_{H^1}\le(\|u\|_{L^4}+\|v\|_{L^4})\|u-v\|_{L^4}
$$
and Rellich's theorem imply that the map $T$ is compact in $\fH^1$. Indeed, if $u_n\to u$ weakly in $\fH^1$, then $u_n\to u$ strongly in $L^4(\bT^3)$ by the Rellich compactness theorem, and the inequality above with $v=u_n$ shows that
$Tu_n\to Tu$ strongly in $\fH^1$. On the other hand, for each $u\in\fH^1$, the equation $\cF_\a(u)=(-\nu\Dlt)^{-1}f_s$ is equivalent to
$$
\Div u=0\,,\qquad \a(u\cdot\grad)u+\grad p=\nu\Dlt u+f_s\,,\quad x\in\bT^3\,.
$$
Multiplying both sides of the motion equation by $u$ and integrating over $\bT^3$, one finds that
$$
\nu\int_{\bT^3}|\grad u|^2dx\le\int_{\bT^3}f_s\cdot udx\le\|(-\Dlt)^{-1/2}f_s\|_{L^2}\|\grad u\|_{L^2}
$$
so that
$$
\|\grad u\|_{L^2}\le\tfrac1\nu\|(-\Dlt)^{-1/2}f_s\|_{L^2}\,.
$$
Setting $R'=\tfrac1\nu\|(-\Dlt)^{-1/2}f_s\|_{L^2}+1$, we see that $\cF_\a$ maps $\overline{B(0,R')}\subset\fH^1$ into $\fH^1$ and that $u\in\d B(0,R')\subset\fH^1$ implies that $\cF_\a(u)\not=(-\nu\Dlt)^{-1}f_s$
for all $\a\in[0,1]$. Therefore
$$
\hbox{degree}(\cF_1,\overline{B(0,R')},(-\nu\Dlt)^{-1}f_s)=\hbox{degree}(\cF_0,\overline{B(0,R')},(-\nu\Dlt)^{-1}f_s)=1
$$
because $\cF_0$ is the identity. Hence the equation
$$
\cF_1(u)=(-\nu\Dlt)^{-1}f_s\,,
$$
which is equivalent to the steady Navier-Stokes equations in $\bT^3$, has at least one solution in $\fH^1$.
\end{proof}

\smallskip
However, the only solution of the Navier-Stokes-Fourier system with viscous heating (\ref{NSFVH}) satisfying (\ref{Mean0}) is $(0,0)$. Indeed, integrating both sides of the last equation in (\ref{NSFVH}) shows that
$$
\int_{\bT^3}|\grad u+(\grad u)^T|^2(x)dx=0\,.
$$
Hence
$$
\grad u+(\grad u)^T=0\,.
$$
In particular, $\Div u=0$ and
$$
-\Div(\grad u)=-\Dlt u=\Div((\grad u)^T)=\grad(\Div u)=0\,,
$$
so that $u$ is a harmonic vector field on $\bT^3$ satisfying (\ref{Mean0}). Hence $u=0$. Returning to the heat equation in (\ref{NSFVH}), we see that $\th$ is a harmonic function on $\bT^3$, and (\ref{Mean0}) implies that $\th=0$.

\newpage

\bs{3.2. The Boltzmann Equation}

\begin{theorem}
Let $F\equiv F(x,v)$ be a $C^1$ solution of the steady Boltzmann equation
$$
v\cdot\grad_xF+f(x)\cdot\grad_vF=\cC(F)\,,\qquad x\in\bT^3\,,\,\,v\in\bR^3\,,
$$
and assume that $F$ is rapidly decaying in $v$ while $\ln F$ has polynomial growth in $v$ as $|v|\to\infty$.

Then $f$ is a gradient field, while $F$ is a Maxwellian distribution with constant temperature. More precisely, there exists $\Phi\in C^1(\bT^3)$, a vector $u\in\bR^3$, and two constants $\th>0$ and $C\ge 0$ such that
$$
f=-\grad\Phi\,,\quad\hbox{ and } F(x,v)=C\exp\left(-\frac1\th\left(\tfrac12|v-u|^2+\Phi(x)\right)\right)
$$
with 
$$
u\cdot\grad\Phi=0\,.
$$
In particular, if $\Div f=0$, then $f=0$ and $F$ is a uniform Maxwellian.
\end{theorem}

\begin{proof}
The global form of Boltzmann's H theorem shows that
$$
\iint_{\bT^3\times\bR^3}\cC(F)\ln Fdxdv=0\,.
$$
Hence $F(x,v)$ is a local Maxwellian satisfying
$$
v\cdot\grad_xF+f(x)\cdot\grad_vF=\cC(F)=0\,,\quad x\in\bT^3\,,\,\,v\in\bR^3\,.
$$
Setting $F(x,v)=\cM_{(\rho(x),u(x),\th(x))}(v)$, the Boltzmann equation reduces to
$$
\ba
\frac{v\cdot\grad\rho(x)}{\rho(x)}+\frac{(v-u)\otimes v:\grad u(x)}{\th(x)}+\frac{v\cdot\grad\th(x)}{2\th(x)^2}(|v-u(x)|^2-3\th(x))&
\\
=
f(x)\cdot\frac{v-u(x)}{\th(x)}&\,.
\ea
$$
Setting
$$
V:=\frac{v-u(x)}{\sqrt{\th(x)}}\,,
$$
the equality above is recast as
$$
\ba
u(x)\cdot\grad\ln\rho(x)+\th(x)^{-1/2}V\cdot(\th(x)\grad\ln\rho(x)+(u\cdot\grad)u)
\\
+\tfrac12(|V|^2-3)u(x)\cdot\grad\ln\th(x)+V^{\otimes 2}:\grad u(x)
\\
+(|V|^2-3)V\cdot\grad\sqrt{\th(x)}=\th(x)^{-1/2}f(x)\cdot V\,.
\ea
$$
Hence
$$
\grad\sqrt\th=0
$$
so that
$$
\th=\hbox{Const.}
$$
and the equality above reduces to
$$
\ba
u(x)\cdot\grad\ln\rho(x)+\th^{-1/2}V\cdot(\th\grad\ln\rho(x)+(u\cdot\grad)u)+V^{\otimes 2}:\grad u(x)&
\\
=\th^{-1/2}f(x)\cdot V&\,.
\ea
$$
Therefore
$$
\ba
u\cdot\grad\rho=0\,,
\\
\th\grad\ln\rho+(u\cdot\grad)u=f\,,
\\
\grad u+(\grad u)^T=0\,.
\ea
$$
Using the third equation, the second equation is recast as
$$
\th\grad\ln\rho+(u\cdot\grad)u=\th\grad\ln\rho-(\grad u)^Tu=\th\grad\ln\rho-\grad\tfrac12|u|^2=f
$$
so that
$$
f=\grad(\th\ln\rho-\tfrac12|u|^2)
$$
must be a gradient field.

Next
$$
\grad u+(\grad u)^T=0
$$
which implies as above that $u$ is harmonic on $\bT^3$, and therefore is a constant.

We conclude that
$$
f(x)=-\grad\Phi(x)\,,\quad F(x,v)=Ce^{-\frac1\th(\frac12|v-u|^2+\Phi(x))}\quad\hbox{ with }u\cdot\grad\Phi(x)=0\,.
$$
Hence $u=0$ unless $\Phi$ takes its values in an affine space orthogonal to $u$.
\end{proof}


\setcounter{chapter}{4}                           
\setcounter{equation}{0} 

\BS{4. Physical discussion and study of a numerical example}

As we observed in Sec.~3.2, there is no spatially periodic steady solution
of the Boltzmann equation for an external force that is not derived from
a potential. In contrast, the Navier--Stokes equations have a
spatially periodic and steady solution for such an external force.
This discrepancy seems to be contradictory, since the Navier--Stokes
equations are derived from the Boltzmann equation, as shown in Sec.~2.1.
In this section, we will further examine this seemingly contradictory results.


\bs{4.1 A possible physical explanation of the paradox}

Let us consider a gas in a periodic box or in a specularly reflecting box.
As we have seen in Sec.~3.2, the steady solution of the Boltzmann equation
with an external force derived from a potential has the following properties:
the temperature is uniform, the macroscopic flow of the gas vanishes
except for a special case, and the density is distributed according to the
potential (i.e., a stratified gas at rest).
Now, let us consider a time-dependent problem
starting from a given initial state. If the external force has a potential,
the time-dependent solution should approach the steady solution mentioned
above, i.e., the solution without a gas flow and with a density stratification,
in the long-time limit. Then, what will happen when the force
does not have a potential? At the initial stage, a gas flow is caused by
the external force. But, since the force does not have a potential, the
density stratification that blocks the gas flow cannot be formed. Therefore,
the flow remains forever, or at least, for much longer time. The induced
gas flow, in general, has a shear. If we consider the case with relatively
small Knudsen number, this shear gives rise to the viscous heating.
However, because the boundary is periodic or adiabatic (in the case of
specularly reflecting box), the heat generated by the shear in the gas
cannot escape through the boundary. This means that the temperature in the
gas will increase indefinitely. This is the reason why there is
no steady solution for the Boltzmann equation.

On the other hand, the Navier--Stokes equations (\ref{NSF}) are derived in 
the limit where the effect of the viscous heating is negligibly small. Therefore, 
the mass and momentum equations (the first two equations in (\ref{NSF})) 
are decoupled with the energy equation (the last equation) when $\Phi=0$. 
The former equations give a velocity field with a shear, and the latter equation 
gives a constant temperature field. However, if we include the effect of viscous 
heating, as was done in Sec.~2.2, the energy equation is changed to the last 
equation in (\ref{NSFVH}), the mass and momentum equations
being unchanged. Therefore, the flow velocity field is the same, but
the new energy equation does not have a solution for this velocity field.
If we consider a time-dependent version of this energy equation, we can
easily see that the temperature increases indefinitely. In conclusion,
the difference between the Boltzmann equation and the Navier--Stokes
equations is due to the fact that the effect of viscous heating is
neglected in the Navier--Stokes limit with $\Ma=\Kn=\Fr^2=\eps$.


\bs{4.2 Numerical example}

For the purpose of understanding the phenomenon of heating predicted above,
we consider a simple numerical example. Let us consider a gas in a
two-dimensional square box $-1/2 < x_1 < 1/2$, $-1/2 < x_2 < 1/2$ with
periodic condition on each side. We assume that the external force is
of the form $f=(0,\, f_0 \sin 2\pi x_1,\, 0)$, which is divergence free
and does not have a potential. Initially, the gas is in a uniform
equilibrium state at rest with density 1 and temperature 1. We pursue the
time evolution of the solution and observe whether the temperature
increases indefinitely or not.

We analyze the problem mainly using the  Bhatnagar--Gross--Krook (also known as BGK)
model, but some results based on the original Boltzmann equation will also
be presented. Because of the form of the external force, we can assume that
the flow field is spatially one dimensional depending only on $x_1$ and
is periodic in $x_1$ with period 1. Moreover, the external force is symmetric
with respect to $x_1=-1/4$ and $x_1=1/4$, so that we can also assume
the same symmetry for the flow field. Therefore, placing specularly
reflecting boundaries at $x_1=-1/4$ and $x_1=1/4$, we can analyze the
problem in the finite interval $-1/4 \le x_1 \le 1/4$.

Here, we formulate the problem using the BGK model. In the present problem,
the BGK model in an appropriate dimensionless form can be written as
\begin{align}
\partial_t F + v_1 \partial_{x_1} F + f_0 \sin2\pi x_1 \partial_{v_2} F
= \frac{1}{\mathrm{Kn}} \mathcal{C}_{\mathrm{BGK}}(F),
\nonumber
\end{align}
where $F(x_1,t,v)$ is the velocity distribution function including
the time variable $t$, and
\begin{subequations}
\begin{align}
& \mathcal{C}_{\mathrm{BGK}}(F) = \left( \frac{8}{\pi} \right)^{1/2}
\rho \left( \mathcal{M} - F \right),
\nonumber \\
& \mathcal{M} = \mathcal{M}_{(\rho(x_1,t), v(x_1,t), \theta(x_1,t))}
= \frac{\rho(x_1,t)}{[2\pi \theta(x_1,t)]^{3/2}}
\exp \left( - \frac{|v - u(x_1,t)|^2}{2\theta(x_1,t)} \right),
\nonumber \\
& \rho(x_1,t) = \int_{{\bm{\mathrm{R}}}^3} F dv,
\nonumber \\
& u_i(x_1,t) = \frac{1}{\rho} \int_{{\bm{\mathrm{R}}}^3} v_i F dv,
\qquad (i=1,\, 2,\, 3;\, u_3=0)
\nonumber \\
& \theta(x_1,t) = \frac{1}{3\rho} \int_{{\bm{\mathrm{R}}}^3} |v - u|^2 F dv.
\nonumber
\end{align}
\end{subequations}
Here, $\mathrm{Kn}$ is the Knudsen number, i.e., the mean free path in the
initial equilibrium state at rest divided by the length of the period (note
that $\mathrm{Kn}$ is denoted by $\epsilon$ in Sec.~2 because the limit as
$\epsilon \to 0$ is discussed there).
The factor $(8/\pi)^{1/2}$ appears because of the manner of
nondimensionalization used here in consistency with the form of the
Boltzmann equation in earlier sections. The specular reflection condition at
$x_1=-1/4$ and $1/4$ is as follows:
\begin{align}
F(\pm 1/4, t, v) = F(\pm 1/4, t, Rv),\quad \mathrm{for}\;\; \mp v_1>0,
\nonumber
\end{align}
where $R$ is the reflection operator: $Rv=(-v_1, v_2, v_3)$. The initial
condition is given by
\begin{align}
F(x_1, 0, v) = \mathcal{M}_{(1, 0, 1)}.
\nonumber
\end{align}
We solve this initial-boundary-value problem by the finite-difference method.

\begin{figure}
\centering
\includegraphics[scale=0.95]{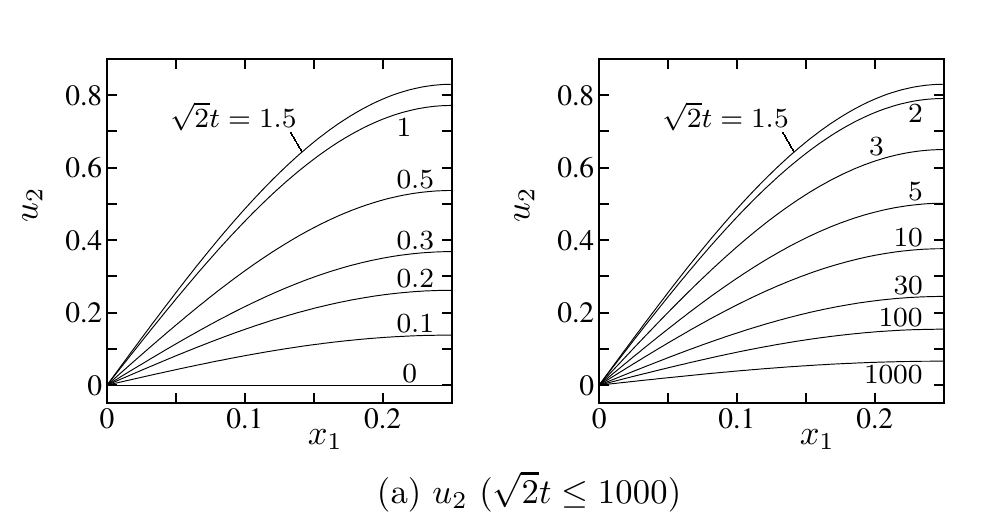} \\
\vspace*{1mm}
\includegraphics[scale=0.95]{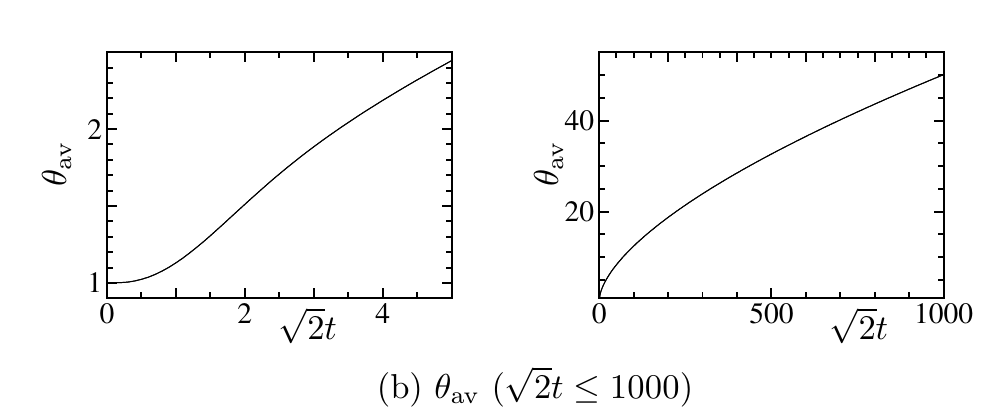} \\
\vspace*{1mm}
\includegraphics[scale=0.95]{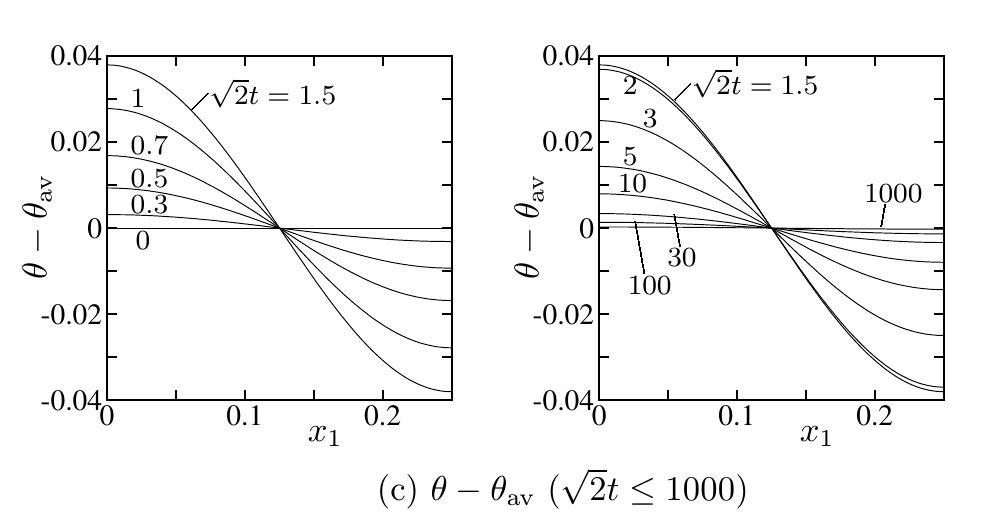} \\
\vspace*{1mm}
\caption{BGK model with $\Kn=0.1$ and $f_0=2$}
\label{fig4.1}
\end{figure}

\addtocounter{figure}{-1}
\begin{figure}
\centering
\includegraphics[scale=0.95]{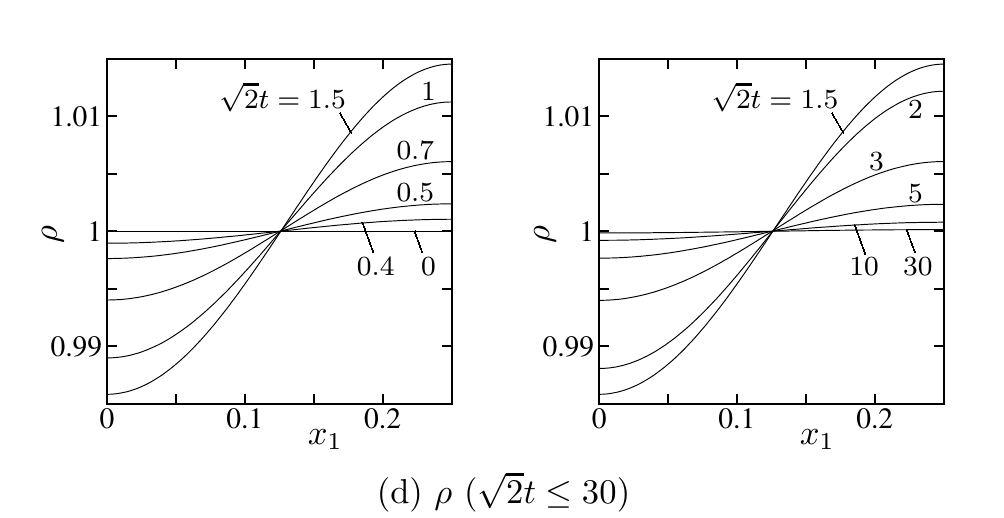} \\
\vspace*{1mm}
\includegraphics[scale=0.95]{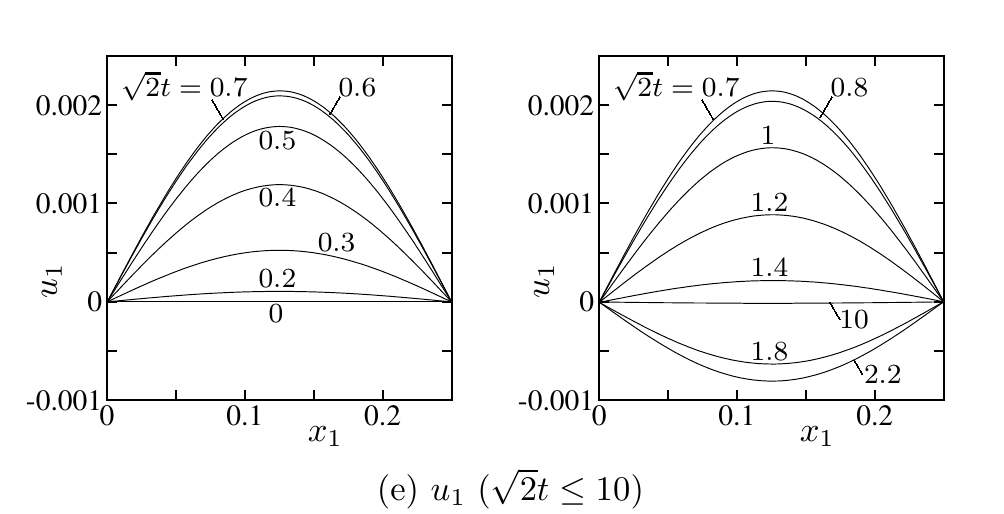} \\
\vspace*{1mm}
\includegraphics[scale=0.95]{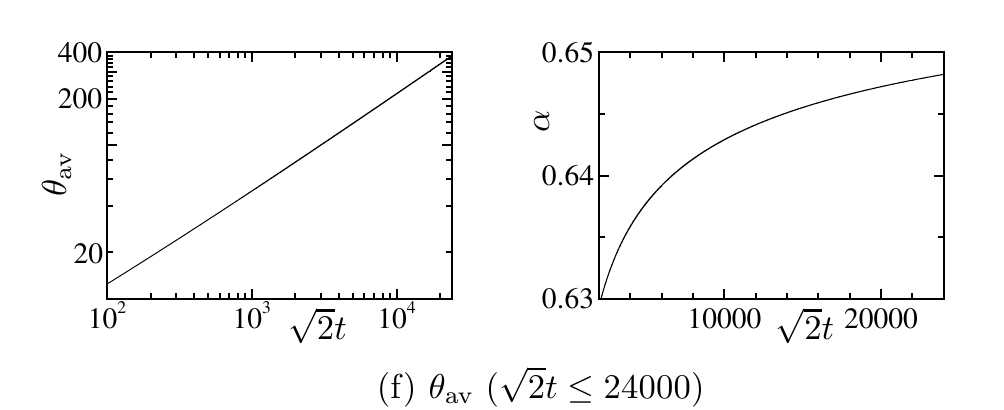} \\
\vspace*{1mm}
\includegraphics[scale=0.95]{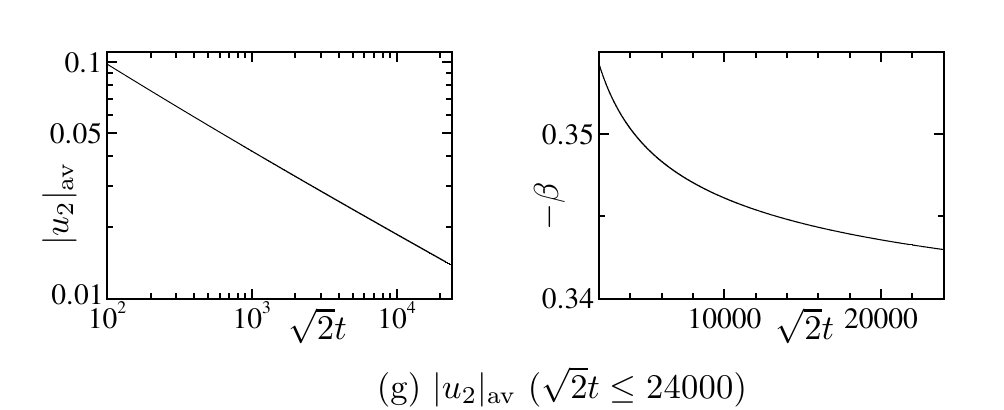} \\
\vspace*{1mm}
\caption{Continued}
\end{figure}

Now we show some of the numerical results for $\mathrm{Kn}=0.1$ and
for $f_0=2$ (Fig.~\ref{fig4.1}) and $0.2$ (Fig.~\ref{fig4.2}).

Figure \ref{fig4.1}(a) shows the profile of the $x_2$ component of the
flow velocity $u_2$ from $t=0$ to $1000/\sqrt{2}$ in the half interval 
$0 \le x_1 \le 1/4$. Since $u_1$ and $u_2$ are odd functions of 
$x_1$, and $\rho$ and $\theta$ are even functions of $x_1$, we show 
the profiles of these quantities only in
the half interval here and in what follows. The sinusoidal external
force parallel to the $x_2$ direction induces $u_2$ at the very early
stage, but $u_2$ starts decreasing after $t=1.5/\sqrt{2}$.
Figure \ref{fig4.1}(b)
shows the time evolution of the average temperature
$\theta_{av}(t)=2\int_{-1/4}^{1/4} \theta(x_1, t)dx_1$ until
$t=1000/\sqrt{2}$, and Fig.~\ref{fig4.1}(c) the corresponding evolution
of the profile of the deviation
$\theta-\theta_{av}$. The $\theta$ becomes nonuniform at the early stage
but tends to become uniform after $t=1.5/\sqrt{2}$. On the other hand,
$\theta_{av}$ increases and reaches 50 at $t=1000/\sqrt{2}$. As shown in
Fig.~\ref{fig4.1}(d), the density $\rho$ is nonuniform only at the very early
stage and is almost uniform (i.e., almost $\rho=1$) at $t=30/\sqrt{2}$.
Corresponding to the nonuniformity of the density, the flow-velocity
component $u_1$ perpendicular to the external force arises at the
very early stage [Fig.~\ref{fig4.1}(e)], but its magnitude is very small
and practically vanishes at $t=10/\sqrt{2}$. Figures \ref{fig4.1}(f)
and \ref{fig4.1}(g) show the long-time
behavior of $\theta_{av}(t)$ and the average speed,
$|u_2|_{av}(t) = 2\int_{-1/4}^{1/4} |u_2(x_1, t)| dx_1$,
up to $t=24000/\sqrt{2}$.
The left figures show the double-logarithmic plot of $\theta_{av}$
versus $t$ and that of $|u_2|_{av}$ versus $t$. In the right figures,
the gradients of the curves in the left figures, i.e.,
$\alpha=d \ln \theta_{av}/d \ln t$ and $\beta=d \ln |u_2|_{av}/d \ln t$
are plotted. If $\alpha$ and $\beta$ approach constant values,
say $\alpha_0$ and $\beta_0$, respectively, then we have the long-time
behavior as $\theta_{av} \approx C_\theta t^{\alpha_0}$ and
$|u_2|_{av} \approx C_u t^{\beta_0}$ with positive constants $C_\theta$ and
$C_u$. From Figs.~\ref{fig4.1}(f) and \ref{fig4.1}(g), it is still not clear
whether $\alpha$ and
$\beta$ converge to finite values or not. But, if it is the case,
it is likely that $\alpha_0 \approx 0.66$ and
$\beta_0 \approx -0.34 \approx -(1-\alpha_0)$.

Figures \ref{fig4.2}(a)--\ref{fig4.2}(g) show the behavior, corresponding
to Figs.~\ref{fig4.1}(a)--\ref{fig4.1}(g),
for a weaker external force ($f_0=0.2$). The tendency of the time evolution
of the solution is similar to Fig.~\ref{fig4.1}. However, since the magnitude
of the force is $1/10$, the resulting flow and the temperature rise are
smaller. As Fig.~\ref{fig4.2}(a) shows, the flow speed $|u_2|$, which is
smaller by one order of
magnitude, takes the maximum at around $t=4/\sqrt{2}$ and decreases more slowly
than in Fig.~\ref{fig4.1}(a). Figure \ref{fig4.2}(b) shows that the increase
of $|\theta_{av}|$
is much slower compared to Fig.~\ref{fig4.1}(b). One sees from
Figs.~\ref{fig4.2}(c)--\ref{fig4.2}(e) that
the nonuniformity of $\theta$ and $\rho$ and the magnitude of $u_1$ are
smaller by two orders of magnitude. In Figs.~\ref{fig4.2}(f) and
\ref{fig4.2}(g), we show the
long-time behavior of $\theta_{av}$ and $|u_2|$ up to an extremely large
time, $t=336000/\sqrt{2}$. As in Figs.~\ref{fig4.1}(f) and \ref{fig4.1}(g),
the left figures
are the $\log-\log$ plots, and the right figures are their gradients
$\alpha$ and $\beta$. Even at such a large time, it is not clear whether
or not $\alpha$ and $\beta$ converge to constants. But, if they converge,
the values would not differ much from the case of $f_0=2$
[cf.~Figs.~\ref{fig4.1}(f) and \ref{fig4.1}(g)].

\begin{figure}
\centering
\includegraphics[scale=0.95]{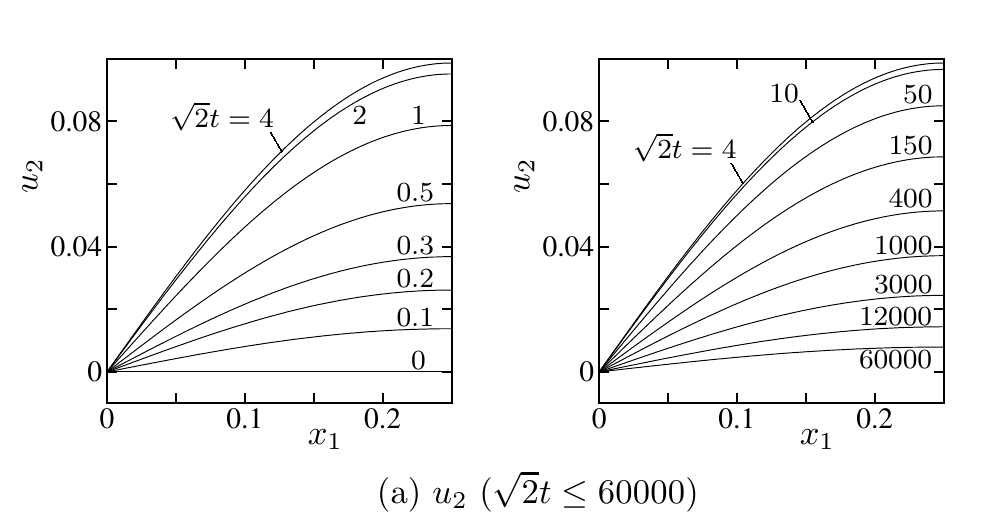} \\
\vspace*{1mm}
\includegraphics[scale=0.95]{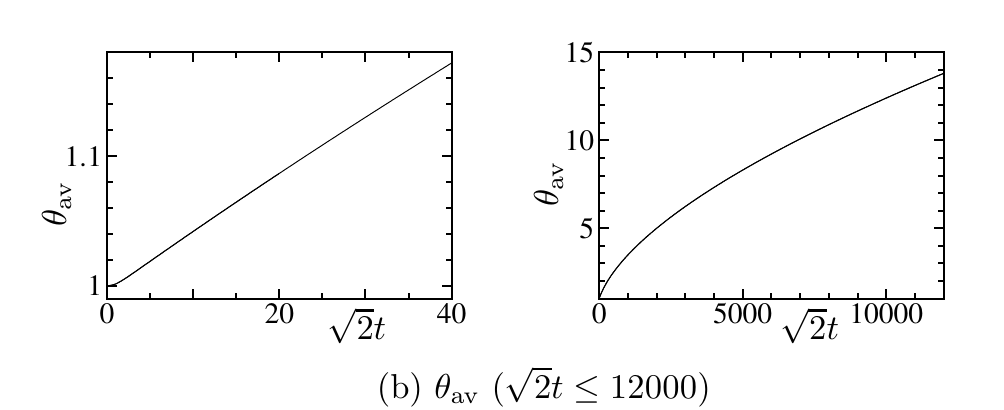} \\
\vspace*{1mm}
\includegraphics[scale=0.95]{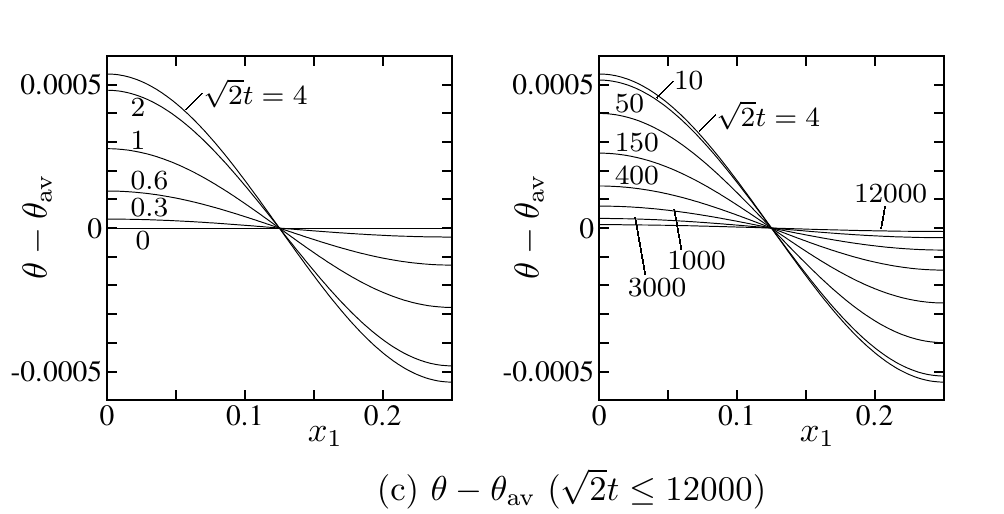} \\
\vspace*{1mm}
\caption{BGK model with $\Kn=0.1$ and $f_0=0.2$}
\label{fig4.2}
\end{figure}

\addtocounter{figure}{-1}
\begin{figure}
\centering
\includegraphics[scale=0.95]{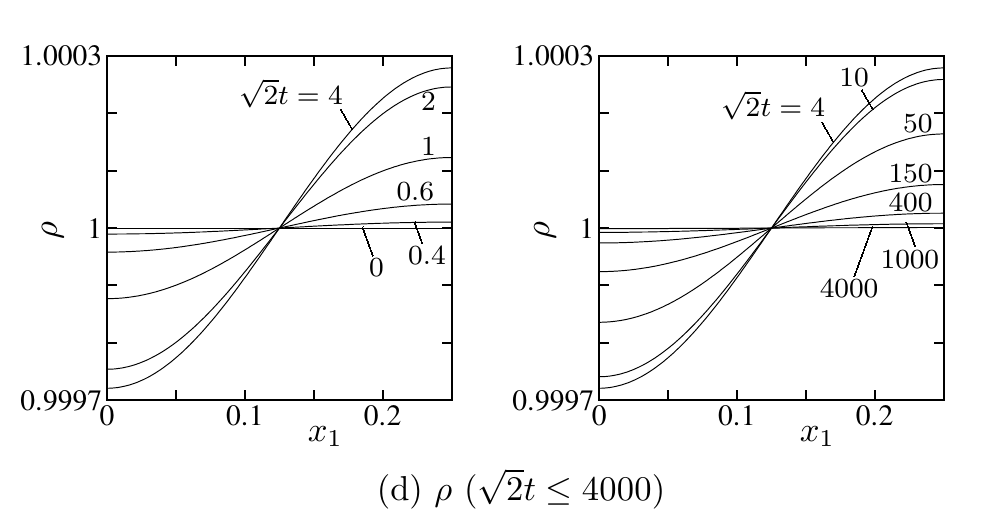} \\
\vspace*{1mm}
\includegraphics[scale=0.95]{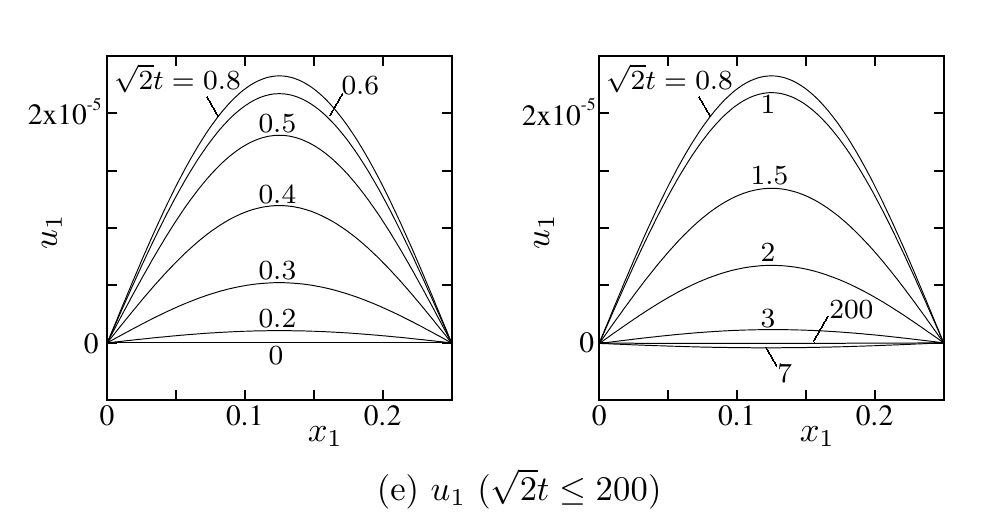} \\
\vspace*{1mm}
\includegraphics[scale=0.95]{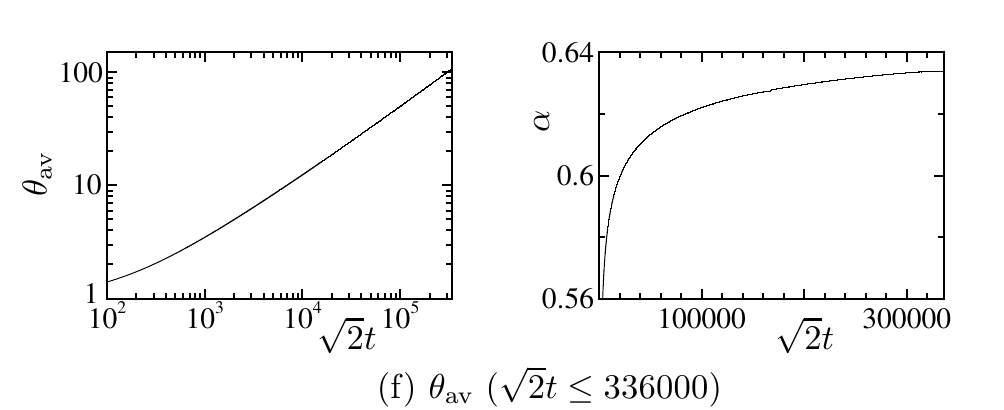} \\
\vspace*{1mm}
\includegraphics[scale=0.95]{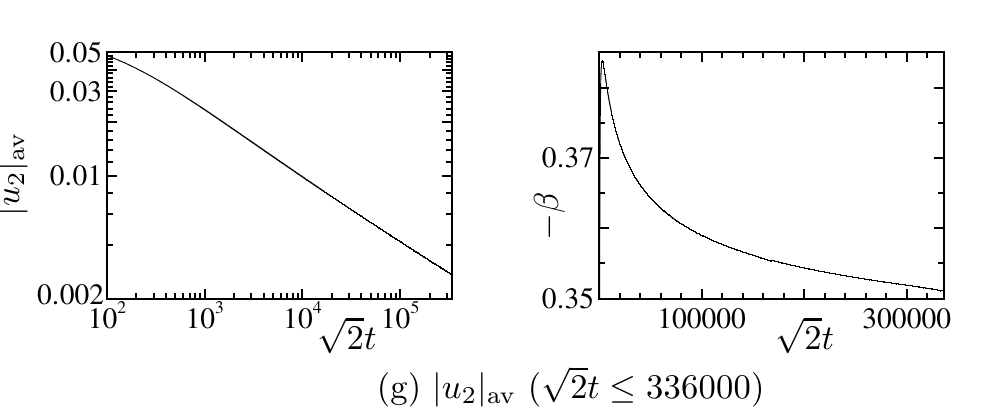} \\
\vspace*{1mm}
\caption{Continued}
\end{figure}

\begin{figure}
\centering
\includegraphics[scale=0.95]{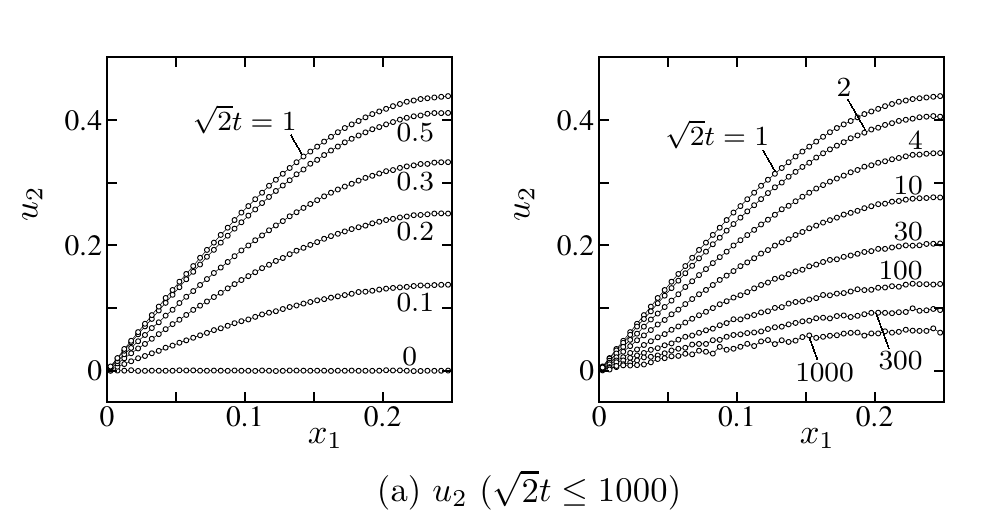} \\
\vspace*{1mm}
\includegraphics[scale=0.95]{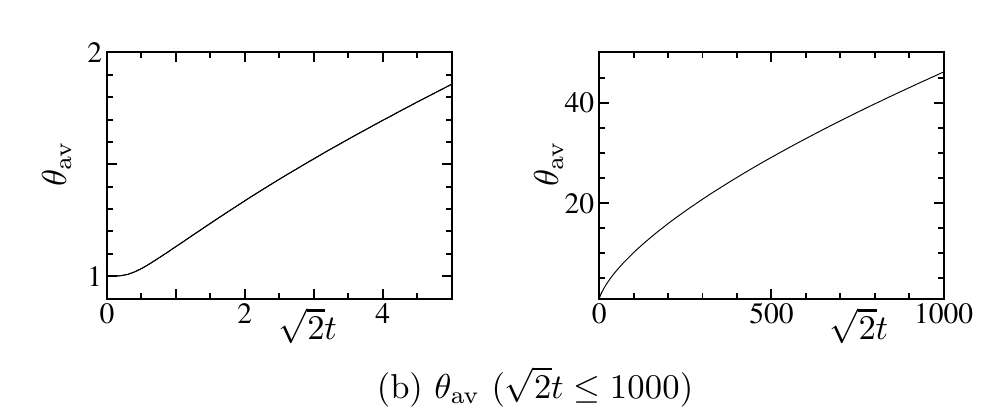} \\
\vspace*{1mm}
\includegraphics[scale=0.95]{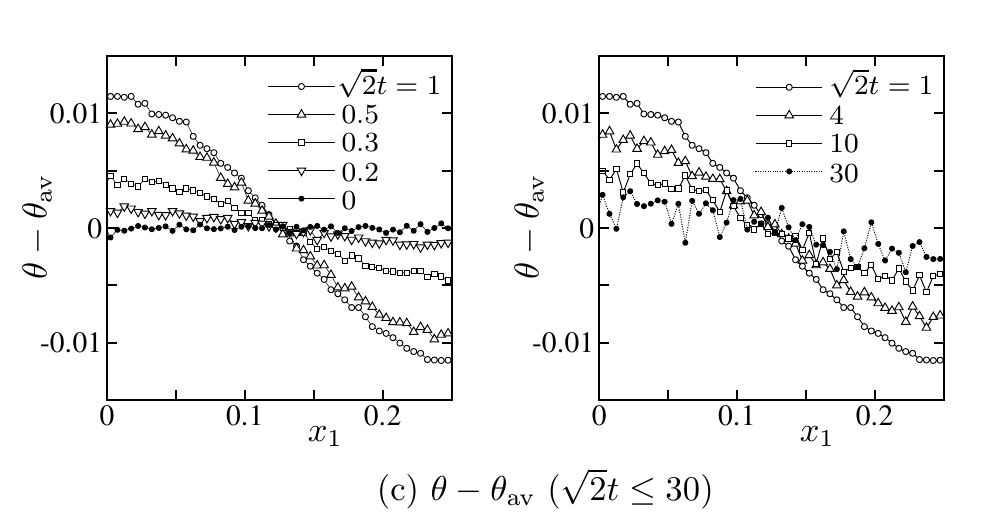} \\
\vspace*{1mm}
\caption{Boltzmann equation with $\Kn=0.1$ and $f_0=2$. 
The ensemble average over 96 independent runs is shown. 
The $\bar{\beta}(t)$ is the time average of $\beta$ 
over the interval $[t-500/\sqrt{2},t]$.}
\label{fig4.3}
\end{figure}

\addtocounter{figure}{-1}
\begin{figure}
\centering
\includegraphics[scale=0.95]{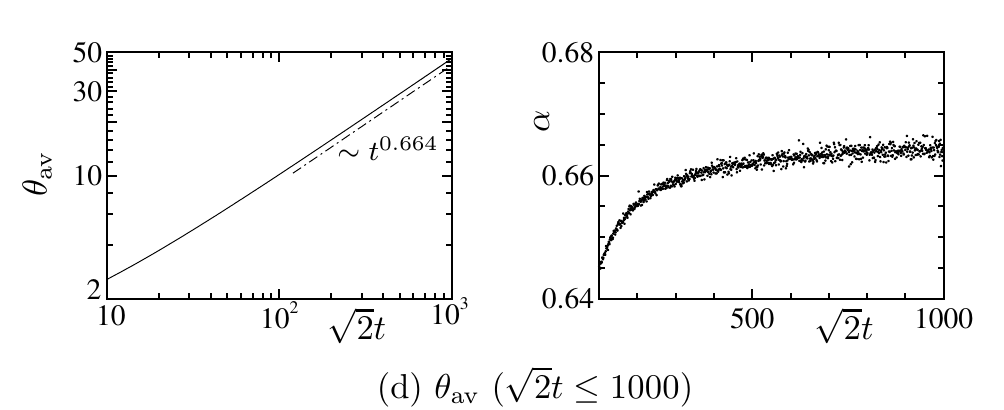} \\
\vspace*{1mm}
\includegraphics[scale=0.95]{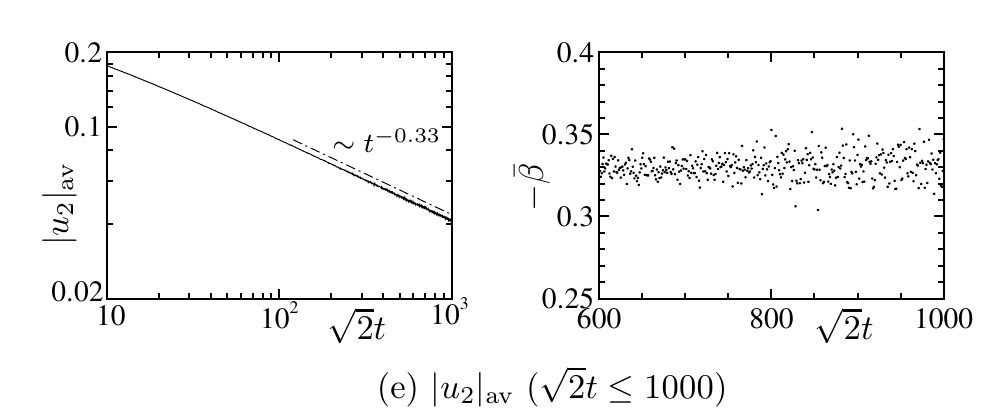} \\
\vspace*{1mm}
\caption{Continued}
\end{figure}

\begin{figure}
\centering
\includegraphics[scale=0.95]{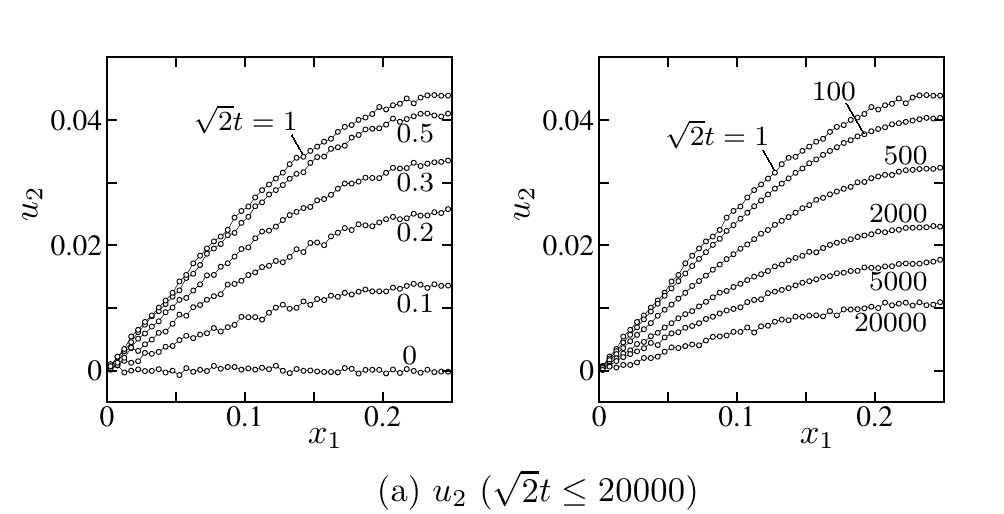} \\
\vspace*{1mm}
\includegraphics[scale=0.95]{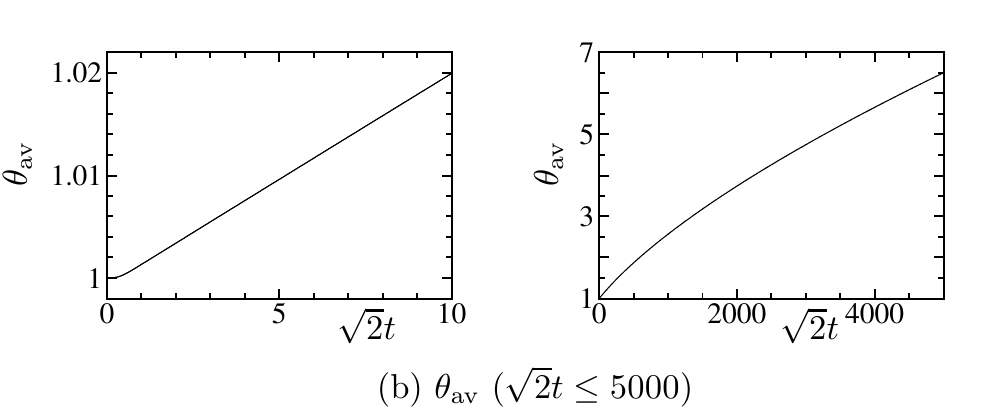} \\
\vspace*{1mm}
\includegraphics[scale=0.95]{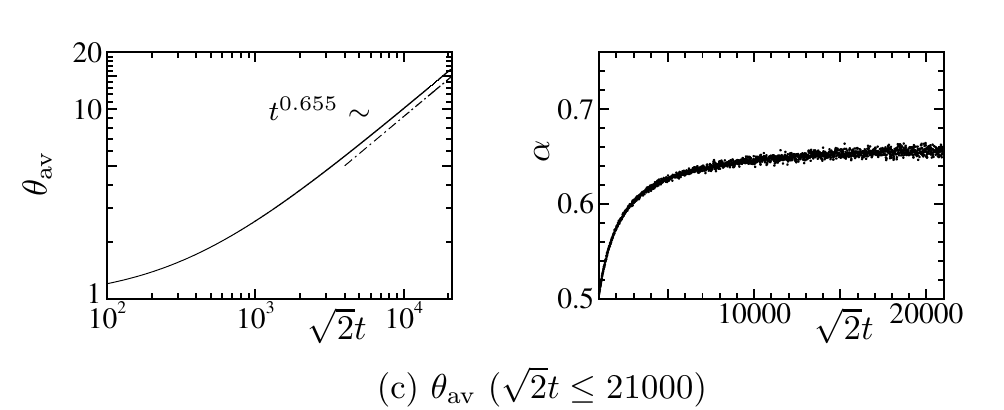} \\
\vspace*{1mm}
\includegraphics[scale=0.95]{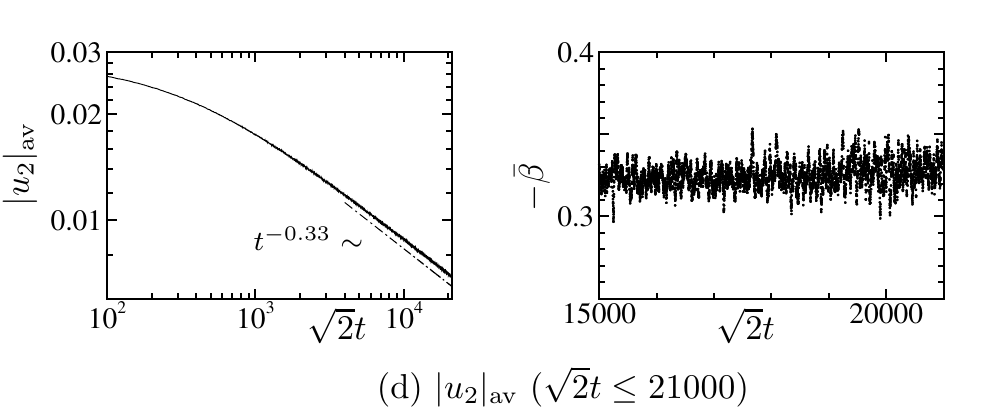} \\
\vspace*{1mm}
\caption{Boltzmann equation with $\Kn=0.1$ and $f_0=0.2$.
The result shown here is based on the ensemble average over $N$ independent
runs and time average over the interval $[t-C/\sqrt{2},\, t+C/\sqrt{2}]$,
where $N$ and $C$ are chosen appropriately depending on $t$. The details
are given in the main text.
The $\bar{\beta}$ is the average of $\beta$ over the time
interval $[t-10000/\sqrt{2},\, t]$.
}
\label{fig4.4}
\end{figure}

Finally, we present some results based on the Boltzmann equation. Let us
replace the right-hand side, $(1/\mathrm{Kn})\mathcal{C}_{\mathrm{BGK}}(F)$,
of the basic equation with $(1/\mathrm{Kn})\mathcal{C}(F)$, where
$\mathcal{C}$ is the dimensionless Boltzmann collision operator for hard-sphere
molecules, as in Sec.~1. The mean free path used here to define $\Kn$ is
$(\sqrt{2}\pi d^2n_0)^{-1}$ with $d$ the diameter of a molecule and $n_0$ 
the reference molecular number density, which is related to $\ell_0$ in Sec.~1 
[cf. (\ref{DimLessB})] as $(2\sqrt{2}\pi)^{-1}\ell_0$.
We employ the direct simulation Monte Carlo
(DSMC) method as the solution method. The method is a particle and stochastic
one, so that it contains the inherent statistical fluctuations. For steady
problems, we can take the time average over a long interval of time to
reduce the fluctuations. But the time-averaging does not work for
time-dependent problems as the present one. The only possible way to reduce
them is to perform many independent runs and take an ensemble average over
the runs. The method is also not appropriate for describing small quantities
because they are hidden in the fluctuations. In fact, it is impossible to
obtain $\theta-\theta_{av}$, $\rho$, and $u_1$. Nevertheless, we show in
Figs.~\ref{fig4.3} and \ref{fig4.4} the results for the same cases as
Figs.~\ref{fig4.1} and \ref{fig4.2}, respectively.
That is, Fig.~\ref{fig4.3} is for $\mathrm{Kn}=0.1$, $f_0=2$, and
Fig.~\ref{fig4.4} for $\mathrm{Kn}=0.1$, $f_0=0.2$. Figures \ref{fig4.3}(a)
and \ref{fig4.4}(a) correspond respectively to Figs.~\ref{fig4.1}(a) and
\ref{fig4.2}(a); Figs. \ref{fig4.3}(b) and \ref{fig4.4}(b) to
Figs.~\ref{fig4.1}(b) and \ref{fig4.2}(b);
Fig.~\ref{fig4.3}(c) to Fig.~\ref{fig4.1}(c)
(it is impossible to obtain the corresponding figure
for $f_0=0.2$);  Figs. \ref{fig4.3}(d) and \ref{fig4.4}(c) to
Figs.~\ref{fig4.1}(f) and \ref{fig4.2}(f); and
Figs.~\ref{fig4.3}(e) and \ref{fig4.4}(d) to Figs.~\ref{fig4.1}(g) and
\ref{fig4.2}(g). Figure \ref{fig4.3} shows the result of
the ensemble average over 96 independent runs. In addition, in
Fig.~\ref{fig4.3}(e),
$\bar{\beta}$, which is the average of $\beta$ over the time interval
$[t-500/\sqrt{2},\, t]$, is shown instead of $\beta$ itself, since it is
impossible to obtain the plot of $\beta$ in a reasonable form.
In Fig.~\ref{fig4.4},
we show the result based on the ensemble average over 96 independent runs for
$0 \le t \le 100/\sqrt{2}$, that based on the ensemble average over 24
independent runs and the time average over the interval
$[t-5/\sqrt{2},\, t+5/\sqrt{2}]$ for $100/\sqrt{2} \le t \le 5000/\sqrt{2}$,
and that based on the ensemble average over 12
independent runs and the time average over the interval
$[t-10/\sqrt{2},\, t+10/\sqrt{2}]$ for $5000/\sqrt{2} \le t$.
The $\bar{\beta}$ in Fig.~\ref{fig4.4}(d) is the average of $\beta$ over the
time interval $[t-10000/\sqrt{2},\, t]$. The computed time in Fig.~\ref{fig4.3}
(or Fig.~\ref{fig4.4}) is much shorter than that in Fig.~\ref{fig4.1}
(or Fig.~\ref{fig4.2}) because the
longer computation is practically impossible. However, the convergence
of $\alpha$ seems to be faster for hard-sphere molecules. It is likely
that $\alpha$ tends to approach about $0.665$ in Fig.~\ref{fig4.3}(d) and about
$0.66$ in Fig.~\ref{fig4.4}(c). Therefore, irrespective of the model of the
collision term, we have the asymptotic behavior like
$\theta_{av} \approx C_\theta t^{0.66}$ for large $t$.


\bs{4.3 Interpretation of the numerical results}

Next, in order to understand the slow increase of the temperature and the
slow decrease of the flow speed, we try a rough discussion on the basis
of the compressible Navier--Stokes equations.
The numerical results show that the flow is almost unidirectional, i.e.,
$u \approx (0,u_2, 0)$. Therefore, let us assume that the flow is
unidirectional, i.e., $u= (0,u_2, 0)$ and the problem in spatially
one-dimensional ($\partial/\partial x_2=\partial/\partial x_3=0$).
Then, the compressible Navier--Stokes equations
reduce to the following equations:
\begin{subequations}\label{CNS}
\begin{align}
& \frac{\partial \rho}{\partial t} = 0,
\nonumber \\
& \frac{\partial p}{\partial x_1} = 0,
\nonumber \\
& \frac{\partial u_2}{\partial t} = \frac{1}{\rho}
\frac{\partial}{\partial x_1}
\left( \mu(\theta) \frac{\partial u_2}{\partial x_1} \right)
+ g_0 \sin 2\pi x_1,
\nonumber \\
& \frac{\partial \theta}{\partial t}
= \frac{2}{3} \frac{1}{\rho} \frac{\partial}{\partial x_1}
\left( \kappa(\theta) \frac{\partial \theta}{\partial x_1} \right)
+ \frac{2}{3} \frac{\mu(\theta)}{\rho} 
\left( \frac{\partial u_2}{\partial x_1} \right)^2,
\nonumber
\end{align}
\end{subequations}
where a suitable nondimensionalization has been made, and
$g_0$ is a constant. In addition, 
$\mu(\theta)$ and $\kappa(\theta)$ are, respectively, dimensionless
forms of the viscosity and thermal conductivity and are
functions of $\theta$. Corresponding to the initial-boundary-value problem
of the BGK model solved numerically, the above equations should be
considered in the interval $-1/4 < x_1 < 1/4$ with the Neumann conditions
\begin{align}
& \frac{\partial u_2}{\partial x_1} = 0,\qquad
\frac{\partial \theta}{\partial x_1} = 0,\qquad
\mathrm{at}\;\; x_1=\pm \frac{1}{4}.
\nonumber
\end{align}

We should keep in mind that the compressible Navier--Stokes equations listed
above hold only approximately because $u_1$ is not exactly zero in
the numerical solution based on the BGK model.
Now, we assume that $\rho \approx \mathrm{const}$ and
$\partial u_2/\partial t$ is negligibly small in the third equation.
Integrating this equation with respect to $x_1$ and taking into account
the boundary condition, we have
$\partial u_2/\partial x_1 \approx (g_0 \rho/2\pi \mu) \cos 2\pi x_1$.
We insert the
expression of $\partial u_2/\partial x_1$ in the forth equation and integrate
it with respect to $x_1$ from $x_1=-1/4$ to $1/4$ assuming that
$\rho \approx \mathrm{const}$ and $\mu(\theta) \approx \mu(\theta_{av})$.
Then, we obtain
\begin{align}
\mu(\theta_{av}) \frac{\partial \theta_{av}}{\partial t}
\approx \frac{g_0^2 \rho}{3 (2\pi)^2}.
\nonumber
\end{align}
Suppose that $\mu(\theta_{av}) = C_\mu \theta_{av}^\delta$. Then, for
the initial condition $\theta=1$ at $t=0$, we obtain
$\theta_{av}^{1+\delta} \approx C_1 t + 1$ with
$C_1 = (1+\delta) g_0^2 \rho/12\pi^2 C_\mu$,
or for large $t$,
\begin{align}
& \theta_{av} \approx (C_1 t + 1)^{\frac{1}{1+\delta}}
\approx C_1^{\frac{1}{1+\delta}} t^{\frac{1}{1+\delta}}.
\nonumber
\end{align}
Since $\partial \theta/\partial t=O( \partial \theta_{av}/\partial t )
= O(t^{-\delta/(1+\delta)})$ and
$\mu(\partial u_2/\partial x_1)^2=O(1/\mu(\theta_{av}))
= O(t^{-\delta/(1+\delta)})$, they are small in the energy equation.
If we neglect these terms in the energy equation, we obtain
$\partial \theta/\partial x_1 \approx 0$, i.e.,
$\theta \approx \theta_{av}$. Then, from the expression
$\partial u_2/\partial x_1 \approx (g_0 \rho/2\pi \mu) \cos 2\pi x_1$
and the boundary condition,
$u_2$ is obtained as $u_2 \approx (g_0 \rho/4\pi^2 \mu) \sin 2\pi x_1$.
To summarize, we obtain
\begin{align}
\theta \approx C_\theta t^{\frac{1}{1+\delta}}, \qquad
u_2 \approx C_u t^{-\frac{\delta}{1+\delta}} \sin 2\pi x_1,
\qquad (\text{for}\;\; t \gg 1),
\nonumber
\end{align}
with new constants $C_\theta$ and $C_u$. Because $\delta=1$ for the BGK model,
we have $\theta \approx C_\theta t^{1/2}$ and
$u_2 \approx C_u t^{-1/2} \sin 2\pi x_1$
for large $t$. On the other hand, since $\delta=0.5$ for hard-sphere molecules,
$\theta$ and $u_2$ behave as $\theta \approx C_\theta t^{2/3}$ and
$u_2 \approx C_u t^{-1/3} \sin 2\pi x_1$. Although this conclusion for
hard-sphere molecules is consistent with the numerical result based on the
Boltzmann equation, it does not coincide precisely with the numerical
result based on the BGK model. It is natural because the argument is too
sketchy. However, it provides qualitative explanation
for the slow increase of the temperature and its uniformity and for
the slow decrease of the flow speed and the sinusoidal flow-velocity profile.
More specifically, as the result of the temperature rise caused by the viscous
heating, the viscosity and the thermal conductivity increase. The high
thermal conductivity leads to a uniform $\theta$. On the other hand, the
high viscosity tends to prevent the external force from causing the gas
flow, so that the flow speed decreases as the temperature increases. However,
the decrease of the flow speed results in the decrease of the viscous heating.
The rough estimate based on the compressible Navier--Stokes equations
show that, as time proceeds, the amount of viscous heating decreases, but the
total amount of heat produced from the initial time increases indefinitely.
Thus, the temperature continues to increase indefinitely.


\setcounter{chapter}{5}                           
\setcounter{equation}{0} 

\BS{Conclusion}

We have discussed the numerous similarities between the steady problem for the Boltzmann equation and for the Navier-Stokes-Fourier system. In particular, the presence of the viscous heating term in the temperature equation depends
on the scaling of the external divergence-free force field. However, we have observed a significant difference between the steady Navier-Stokes equation and the steady Boltzmann equation: while the steady Navier-Stokes equation with 
prescribed divergence-free external force always have a solution, either in the periodic setting or in the case of a bounded spatial domain with Dirichlet boundary condition, the steady Boltzmann equation with nonzero divergence-free 
external force field cannot have a nonzero solution. We have proposed a physical explanation for this difference, based on the long time behavior of the evolution Boltzmann equation with nonzero divergence-free external force field in
the periodic setting. Our numerical computations based on the BGK model suggest that the temperature field in the gas increases indefinitely, so that the gas cannot reach a steady state.

While we have presented our results on the hydrodynamic limit of the steady Boltzmann equation in the case of a hard sphere gas, the same results should remain true in the case of hard cutoff potentials (in the sense of Grad).

Finally, our numerical simulations suggest the following problem, which we believe is open at the time of this writing.

\smallskip
\textbf{Problem:} Consider the evolution Boltzmann equation set in the periodic box, with a prescribed nonzero divergence-free external force field (for instance $f(x_1,x_2,x_3)=(0,f_0\sin(2\pi x_1),0)$). What is the asymptotic behavior
of the temperature field averaged in the space variable in the long time limit? In particular, does there exist $\a>0$ such that the average temperature field is asymptotically equivalent to $Ct^\a$ for some $C>0$ as the time variable
$t$ tends to infinity?

\smallskip
\textbf{Note added after publication:} After the present article was published by the Bulletin of the Institute of Mathematics of Academia Sinica, we became aware of the reference \cite{GuoEspoMarra2}. This reference provides a complete
proof of the asymptotic limit described in Theorem 1 at the formal level. The result established in \cite{GuoEspoMarra2} assumes that the source term in the Navier-Stokes-Fourier limiting system is small.


\bigskip
\noindent
\textbf{Acknowledgement.} This work was started during a visit of the first author at Ecole polytechnique, and completed while the second author was visiting Kyoto University. We are grateful to both institutions for their hospitality and
support.


\setcounter{chapter}{6}                           
\setcounter{equation}{0} 

\BS{Appendix: Properties of $A$ and $B$}

\begin{lemma}\lb{L-PtyAB}
The tensor field $A$ and the vector field $B$ defined by the formulas 
$$
A(v):=v\otimes v-\tfrac13|v|^2I\,,\qquad B(v):=\tfrac12(|v|^2-5)v
$$
satisfy the following properties.

\smallskip
\noindent
(1) The orthogonality relations
$$
A\perp\Ker\cL\,,\quad B\perp\Ker\cL\,,\quad A\perp B
$$
hold componentwise in $L^2(\bR^3,Mdv)$.

\noindent
(2) There exists a unique tensor field $\hat A$ and a unique vector field $\hat B$ in $L^2(\bR^3;(1+|v|^2)Mdv)$ such that
$$
\ba
\cL\hat A=A&&\hbox{ and }\hat A\perp\Ker\cL\,,
\\
\cL\hat B=B&&\hbox{ and }\hat B\perp\Ker\cL\,.
\ea
$$

\noindent
(3) The tensor field $\hat A$ and the vector field $\hat B$ are of the form
$$
\hat A(v)=\a(|v|)A(v)\,,\qquad\hat B(v)=\b(|v|)B(v)
$$
for a.e. $v\in\bR^3$.
\end{lemma}

Statement (1) is Lemma 5.3 in \cite{FGBraga}, while statement (3) is Lemma 5.4 in \cite{FGBraga}. Statement (3) has been used systematically in the literature on the Boltzmann equation, referring to \S 7.31 in \cite{CC}. However, the discussion
in \cite{CC} is incomplete. The key argument leading to the structure of $\hat A$ and $\hat B$ in statement (3) is the invariance of the linearized collision operator under the orthogonal group $O_3(\bR)$. It seems that the first complete proof of 
statement (3) following this idea is in \cite{DesvilleFG}. Statement (2) follows from Hilbert's lemma (the fact that the linearized collision operator $\cL$ satisfies the Fredholm alternative).

\begin{lemma}\lb{L-IntPtyAAhatBBhat} The tensor fields $A$ and $\hat A$, and the vector fields $B$ and $\hat B$ satisfy the following properties.

\smallskip
\noindent
(1) For each $i,j=1,2,3$, one has
$$
\la B_iB_j\ra=\de_{ij}\tfrac13\la\tfrac14(|v|^2-5)^2|v|^2\ra=\tfrac52\de_{ij}\,,
$$
and
$$
\la\hat B_iB_j\ra=\de_{ij}\tfrac13\la\tfrac14(|v|^2-5)^2|v|^2\b(|v|)\ra=\ka\de_{ij}\,,
$$
where
$$
\ba
\ka:=&\tfrac1{12}\la(|v|^2-5)^2|v|^2\b(|v|)\ra
\\
=&\tfrac13\la\hat B\cdot B\ra=\tfrac13\la\hat B\cdot\cL\hat B\ra>0\,.
\ea
$$
(2) For each $i,j,k,l=1,2,3$, one has
$$
\ba
\la A_{ij}A_{kl}\ra&=(\de_{ik}\de_{jl}+\de_{il}\de_{jk}-\tfrac23\de_{ij}\de_{kl})\tfrac1{15}\la|v|^4\ra
\\
&=(\de_{ik}\de_{jl}+\de_{il}\de_{jk}-\tfrac23\de_{ij}\de_{kl})
\ea
$$
and
$$
\la\hat A_{ij}A_{kl}\ra=\nu(\de_{ik}\de_{jl}+\de_{il}\de_{jk}-\tfrac23\de_{ij}\de_{kl})
$$
where
$$
\ba
\nu:=&\tfrac1{15}\la|v|^4\a(|v|)\ra
\\
=&\tfrac1{10}\la\hat A:A\ra=\tfrac1{10}\la\hat A:\cL\hat A\ra>0\,.
\ea
$$
(3) For each $i,j,k,l=1,2,3$, one has
$$
\la\hat B_iv_jA_{kl}\ra=\tfrac25\ka(\de_{ik}\de_{jl}+\de_{il}\de_{jk}-\tfrac23\de_{ij}\de_{kl})
$$
and
$$
\la\hat B_iv_j\hat A_{kl}\ra=c(\de_{ik}\de_{jl}+\de_{il}\de_{jk}-\tfrac23\de_{ij}\de_{kl})\,,
$$
with
$$
c:=\tfrac1{10}\la(\hat B\otimes v):\hat A\ra\,.
$$
\end{lemma}

Statement (2) is Lemma 3.4 in \cite{FGBraga} (proved in Appendix 2 of \cite{FGBraga}). See formulas (4.10) and (4.13b) in \cite{BGL2} for statement (1), which is proved by a similar (though simpler) argument as statement (2). (Notice 
the slight difference in normalization in the definitions of $\ka$ in \cite{BGL2}Êand here).

\begin{proof}[Proof of statement (3)] 
By statement (3) in the previous lemma
$$
\ba
\la\hat B_iv_jA_{kl}\ra=&\la\b(|v|)\tfrac12(|v|^2-5)v_iv_jv_kv_l\ra
\\
&-\de_{kl}\la\b(|v|)\tfrac16(|v|^2-5)|v|^2v_iv_j\ra
\\
=&\la\b(|v|)\tfrac12(|v|^2-5)v_iv_jv_kv_l\ra
\\
&-\de_{kl}\la\b(|v|)\tfrac16(|v|^2-5)^2v_iv_j\ra\,,
\ea
$$
where the second equality follows form the fact that $\hat B\perp\Ker\cL$. Then, according to Lemma 4.3 in \cite{FGBraga}, one has
$$
\la\b(|v|)\tfrac12(|v|^2-5)v_iv_jv_kv_l\ra=\l(\de_{ij}\de_{kl}+\de_{ik}\de_{jl}+\de_{il}\de_{jk})
$$
for some $\l\in\bR$, while, by statement (1) of the previous lemma,
$$
\la\b(|v|)\tfrac16(|v|^2-5)^2v_iv_j\ra=\tfrac23\la\hat B_iB_j\ra=\tfrac23\ka\de_{ij}\,.
$$
Hence
$$
\la\hat B_iv_jA_{kl}\ra=\l(\de_{ij}\de_{kl}+\de_{ik}\de_{jl}+\de_{il}\de_{jk})-\tfrac23\ka\de_{ij}\de_{kl}\,,
$$
and since 
$$
(5\l-2\ka)\de_{ij}=\la\hat B_iv_jA_{kk}\ra=\la\hat B_iv_j\Tr(A)\ra=0\,,
$$
we conclude that $\l=\tfrac25\ka$, which immediatly implies the first equality in statement (3). The second equality is obtained in exactly the same manner. Contracting $i,k$ and $j,l$, one finds the announced formula for $c$.
\end{proof}

\begin{lemma}\lb{L-IntBC}
Let $C$ be the tensor field defined by the formula
$$
C(v):=\tfrac12(v\otimes v\otimes v-3v\otimes I)\,.
$$
Then, for each $i,j,k,l=1,2,3$, one has
$$
\la B_i\otimes C_{jkl}\ra=\tfrac12(\de_{ij}\de_{kl}+\de_{ik}\de_{jl}+\de_{il}\de_{jk})\,.
$$
\end{lemma}

\begin{proof}
First
$$
\la B_i\otimes C_{jkl}\ra=\tfrac12\la B_iv_jv_kv_l\ra\,,
$$
since $B\perp\Ker\cL$. Thus
$$
\ba
\tfrac12\la B_iv_jv_kv_l\ra&=\tfrac14\la(|v|^2-5)v_iv_jv_kv_l\ra
\\
&=\mu(\de_{ij}\de_{kl}+\de_{ik}\de_{jl}+\de_{il}\de_{jk})
\ea
$$
for some $\mu\in\bR$, by Lemma 4.3 in \cite{FGBraga}. By contraction on the indices $k,l$
$$
5\mu\de_{ij}=\tfrac12\la B_iv_j|v|^2\ra=\la B_iB_j\ra=\tfrac52\de_{ij}
$$
where the first equality follows from the orthogonality relation $B\perp\Ker\cL$ by statement (1) in Lemma \ref{L-PtyAB}. Hence $\mu=\tfrac12$, and this immediatly implies the desired identity.
\end{proof}

\newpage


\lhead[\small\thepage\fancyplain{}\leftmark\hfil$[$]{}
\rhead[]{\small \fancyplain{}\rightmark\small\thepage} \cfoot{}
\markboth {\hfill{\small\rm\Bauthor}\hfill} {\hfill {\small\rm\Bshorttitle}\hfill}



\renewcommand\bibname{\centerline{\large\bf References}}
\fontsize{9}{11.0pt plus1pt minus .8pt}\selectfont

\end{document}